\documentclass[oneside,english,11pt]{amsart}
\usepackage[T1]{fontenc}
\usepackage[latin9]{inputenc}
\usepackage{amsthm}
\usepackage{amstext}
\usepackage{amssymb}
\usepackage{graphicx}

\makeatletter
\theoremstyle{plain}
\newtheorem{thm}{\protect\theoremname}
  \theoremstyle{definition}
  \newtheorem{defn}[thm]{\protect\definitionname}
  \theoremstyle{plain}
  \newtheorem{lem}[thm]{\protect\lemmaname}
  \theoremstyle{plain}
  \newtheorem{prop}[thm]{\protect\propositionname}
  \theoremstyle{plain}
  \newtheorem{cor}[thm]{\protect\corollaryname}
  \theoremstyle{remark}
  \newtheorem{rem}[thm]{\protect\remarkname}

\usepackage[margin=1in]{geometry}
\usepackage{cases}
\usepackage{hyperref}

\theoremstyle{definition}
\newtheorem{assume}{Assumption}
\newtheorem{constr}{Construction}

\@ifundefined{showcaptionsetup}{}{%
 \PassOptionsToPackage{caption=false}{subfig}}
\usepackage{subfig}
\makeatother

\usepackage{babel}
  \providecommand{\corollaryname}{Corollary}
  \providecommand{\definitionname}{Definition}
  \providecommand{\lemmaname}{Lemma}
  \providecommand{\propositionname}{Proposition}
  \providecommand{\remarkname}{Remark}
\providecommand{\theoremname}{Theorem}

\begin{document}

\title[Parameter Insensitivity in ADMM Preconditioner]{Parameter Insensitivity in ADMM-Preconditioned Solution of Saddle-Point
Problems}

\author{Richard Y. Zhang}

\author{Jacob K. White}

\thanks{Department of Electrical Engineering and Computer Science, Massachusetts
Institute of Technology, Cambridge, MA 02139-4307. Email: \texttt{ryz@mit.edu}
and \texttt{white@mit.edu}. Financial support for this work was provided
in part by the Skolkovo-MIT initiative in Computational Mathematics.}
\begin{abstract}
We consider the solution of linear saddle-point problems, using the
alternating direction method-of-multipliers (ADMM) as a preconditioner
for the generalized minimum residual method (GMRES). We show, using
theoretical bounds and empirical results, that ADMM is made remarkably
insensitive to the parameter choice with Krylov subspace acceleration.
We prove that ADMM-GMRES can consistently converge, irrespective of
the exact parameter choice, to an $\epsilon$-accurate solution of
a $\kappa$-conditioned problem in $O(\kappa^{2/3}\log\epsilon^{-1})$
iterations. The accelerated method is applied to randomly generated
problems, as well as the Newton direction computation for the interior-point
solution of semidefinite programs in the SDPLIB test suite. The empirical
results confirm this parameter insensitivity, and suggest a slightly
improved iteration bound of $O(\sqrt{\kappa}\log\epsilon^{-1})$.
\end{abstract}

\maketitle

\section{Introduction}

\global\long\def\R{\mathbb{R}}
\global\long\def\C{\mathbb{C}}
\global\long\def\S{\mathbb{S}}
\global\long\def\P{\mathbb{P}}
\global\long\def\K{\mathcal{K}}
\global\long\def\L{\mathscr{L}}
\global\long\def\blkdiag{\mathrm{blkdiag}}
\global\long\def\vec{\mathrm{vec}}
\global\long\def\mat{\mathrm{mat}}
\global\long\def\tr{\mathrm{tr}\,}
\global\long\def\AD{\mathrm{AD}}
\global\long\def\GM{\mathrm{GM}}
\global\long\def\OR{\mathrm{OR}}
\global\long\def\gs{\mathrm{GS}}
\global\long\def\nx{n_{x}}
\global\long\def\nz{n_{z}}
\global\long\def\ny{n_{y}}
\global\long\def\lb{\mathrm{lb}}
\global\long\def\prox{\mathrm{prox}}
We consider iteratively solving very large scale instances of the
saddle-point problem
\begin{equation}
\begin{bmatrix}D & 0 & A^{T}\\
0 & 0 & B^{T}\\
A & B & 0
\end{bmatrix}\begin{bmatrix}x\\
z\\
y
\end{bmatrix}=\begin{bmatrix}r_{x}\\
r_{z}\\
r_{y}
\end{bmatrix}\qquad\Leftrightarrow\qquad Mu=r,\label{eq:basic_KKT}
\end{equation}
with data matrices $A\in\R^{\ny\times\nx}$ with $AA^{T}$ invertible,
$B\in\R^{\ny\times\nz}$ with $B^{T}B$ invertible, symmetric positive
definite $D\in\R^{\ny\times\ny}$, and data vectors $r_{x}\in\R^{\nx}$,
$r_{z}\in\R^{\nz}$, $r_{y}\in\R^{\ny}$. Note that the special case
of $A=I$ reduces (\ref{eq:basic_KKT}) to the familiar block $2\times2$
saddle-point structure
\[
\begin{bmatrix}D^{-1} & B\\
B^{T} & 0
\end{bmatrix}\begin{bmatrix}-y\\
z
\end{bmatrix}=\begin{bmatrix}r_{y}-D^{-1}r_{x}\\
-r_{z}
\end{bmatrix}.
\]
Additionally, we assume that efficient solutions (i.e. black-box oracles)
to the two subproblems 
\begin{equation}
\begin{bmatrix}D & A^{T}\\
A & -\beta^{-1}I
\end{bmatrix}\begin{bmatrix}\tilde{x}\\
\tilde{y}
\end{bmatrix}=\begin{bmatrix}\tilde{r}_{x}\\
\tilde{r}_{y}
\end{bmatrix},\qquad\begin{bmatrix}0 & B^{T}\\
B & -\beta^{-1}I
\end{bmatrix}\begin{bmatrix}\tilde{z}\\
\tilde{y}
\end{bmatrix}=\begin{bmatrix}\tilde{r}_{z}\\
\tilde{r}_{y}
\end{bmatrix},\label{eq:oracles}
\end{equation}
are available for a fixed choice of $\beta>0$. 

Saddle-point problems with this structure arise in numerous settings,
ranging from nonlinear optimization to the numerical solution of partial
differential equations (PDEs); the subproblems (\ref{eq:oracles})
are often solved with great efficiency by exploiting application-specific
features. For example, when the data matrices are large-and-sparse,
the smaller saddle-point problems (\ref{eq:oracles}) can admit highly
sparse factorizations, based on nested dissection or minimum degree
orderings~\cite{vanderbei1993symmetric,vanderbei1995symmetric}.
Also, the Schur complements $D+\beta A^{T}A$ and $B^{T}B$ are symmetric
positive definite, and can often be interpreted as discretized Laplacian
operators, for which many fast solvers are available~\cite{trottenberg2000multigrid,stuben2001review,vishnoi2012laplacian,spielman2014nearly}.
In some special cases, a triangular factorization or a diagonalization
may be available analytically~\cite{toh2004solving}. The reader
is referred to~\cite{benzi2005numerical} for a more comprehensive
review of possible applications.

The problem structure has an interpretation of establishing \emph{consensus}
between the two subproblems. To see this, note that (\ref{eq:basic_KKT})
is the Karush--Kuhn--Tucker (KKT) optimality condition associated
with the equality-constrained least-squares problem
\begin{alignat}{2}
 & \underset{x,z}{\text{minimize }}\qquad & \frac{1}{2}x^{T}Dx-r_{x}^{T}x-r_{z}^{T}z,\label{eq:consen2}\\
 & \text{subject to } & Ax+Bz=r_{y},\nonumber 
\end{alignat}
and the solution to (\ref{eq:basic_KKT}) is the unique optimal point.
This problem is easy to solve if one of two variables were held fixed.
For instance, holding $z$ fixed, the minimization of (\ref{eq:consen2})
over $x$ to $\epsilon$-accuracy can be made with just a single call
to the first subproblem in (\ref{eq:oracles}), taking the parameter
to be $\beta\in O(\epsilon^{-1})$. The difficulty of the overall
problem, then, lies entirely in the need for consensus, i.e. for two
independent minimizations to simultaneously satisfy a single equality
constraint.

The alternating direction method-of-multipliers (ADMM) is a popular
first-order method widely used in signal processing, machine learning,
and related fields, to solve consensus problems like the one posed
in (\ref{eq:consen2}); cf.~\cite{boyd2011distributed} for an extensive
review. Each ADMM iteration calls the subproblems in (\ref{eq:oracles}),
with $\beta$ serving as the step-size parameter for the underlying
gradient ascent. Under the assumptions on the data matrices stated
at the start of the paper, ADMM converges at a linear rate (with error
scaling $O(e^{-k})$ at the $k$-th iteration), starting from any
initial point.

The choice of the parameter $\beta$ heavily influences the effectiveness
of ADMM. Using an optimal choice~\cite{giselsson2014diagonal,nishihara2015general,francca2015explicit,ghadimi2015optimal},
the method is guaranteed converge to an $\epsilon$-accurate solution
in
\begin{equation}
O(\sqrt{\kappa}\log\epsilon^{-1})\text{ iterations},\label{eq:optim_iter_cnt}
\end{equation}
where $\kappa$ is the condition number associated with the rescaled
matrix $\tilde{D}=(AD^{-1}A^{T})^{-1}$. This bound is asymptotically
optimal, in the sense that the square-root factor cannot be improved~\cite[Thm. 2.1.13]{nesterov2004introductory}. 

Unfortunately, explicitly estimating the optimal parameter choice
can be challenging. Picking any arbitrarily value, say $\beta=1$,
often results in convergence that is so slow as to be essentially
stagnant, even on well-conditioned problems~\cite{ghadimi2015optimal,nishihara2015general}.
A heuristic that works well in practice is to adjust the parameter
after each iteration, using a rule-of-thumb based on keeping the primal
and dual residuals within the same order of magnitude~\cite{he2000alternating,wang2001decomposition,boyd2011distributed}.
However, varying the value of $\beta$ between iterations can substantially
increase the cost of solving the subproblems in (\ref{eq:oracles}).

\subsection{Main results}

\begin{figure}
\begin{centering}
\subfloat[]{\begin{centering}
\includegraphics[width=0.45\columnwidth]{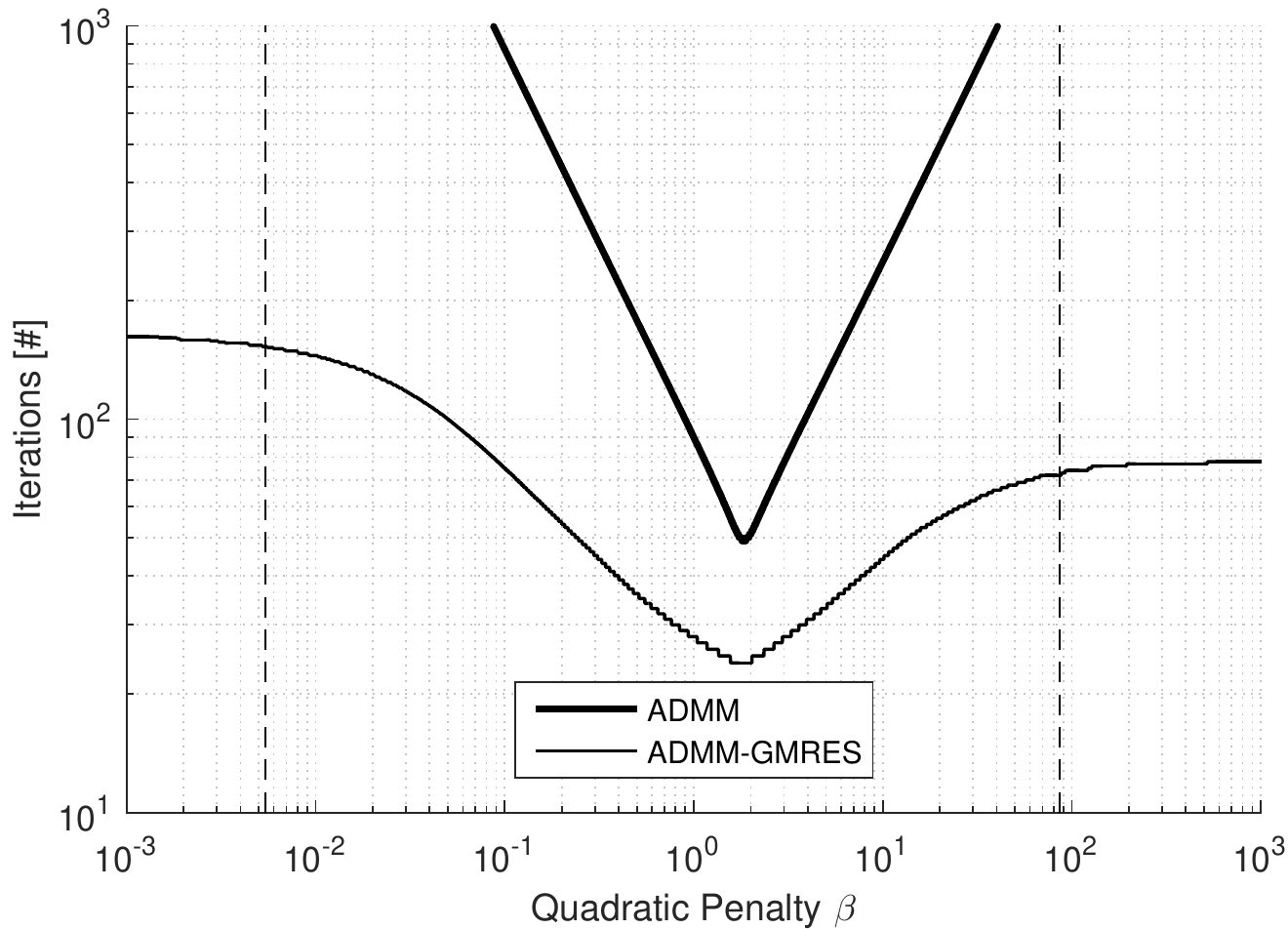}
\par\end{centering}

}\subfloat[]{\begin{centering}
\includegraphics[width=0.45\columnwidth]{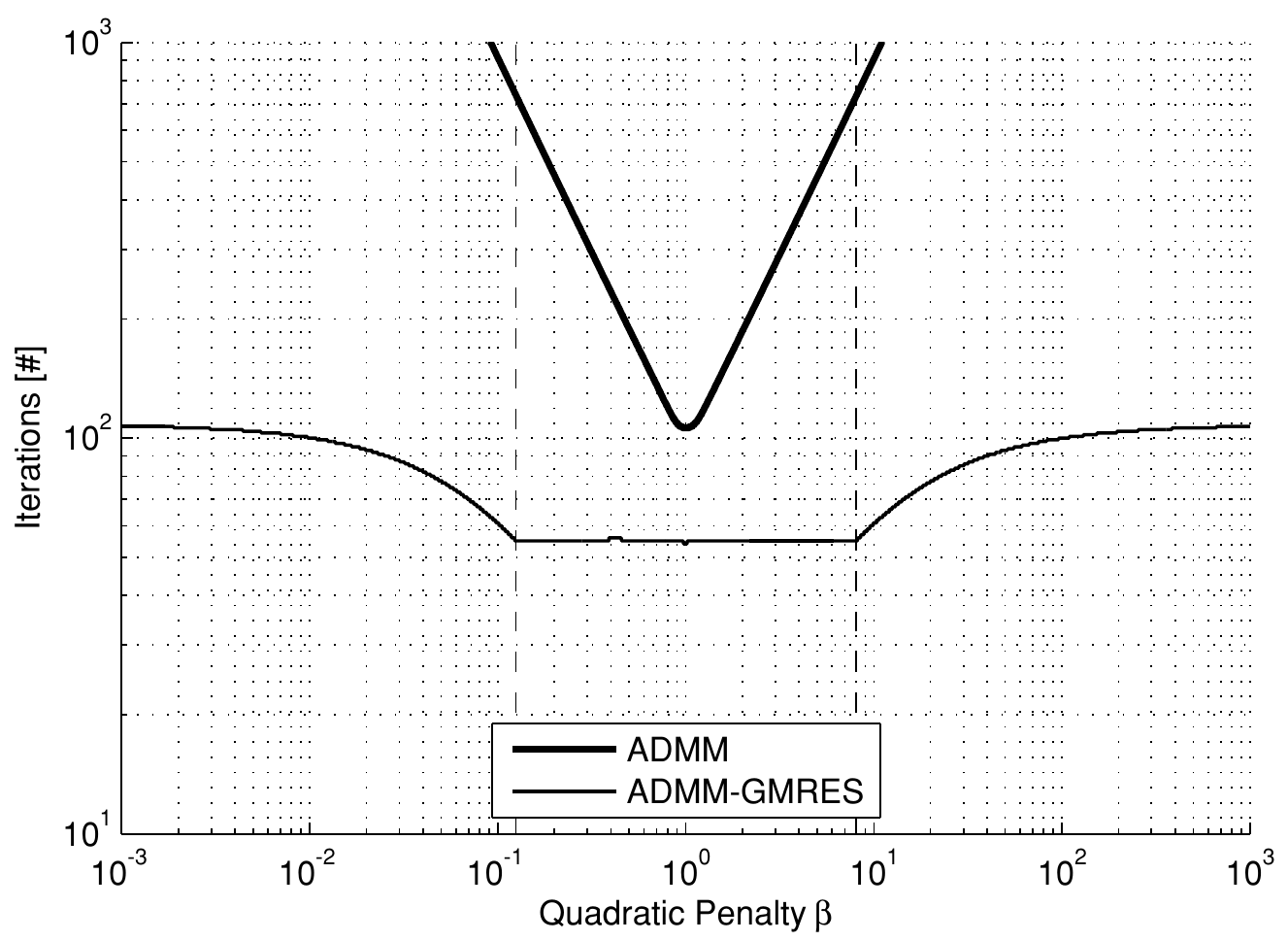}
\par\end{centering}

}
\par\end{centering}

\protect\caption{\label{fig:iter_comparison}Convergence of GMRES-accelerated ADMM
and regular ADMM with varying $\beta$, and error convergence tolerance
$\epsilon=10^{-6}$: (a) randomly generated problem with dimensions
$\protect\ny=10^{3}$, $\protect\nx=2\times10^{3}$, $\protect\nz=300$.
The vertical lines mark $m=5.4\times10^{-3}$ and $\ell=86$ for $\sqrt{m\ell}=0.68$
and $\kappa=\ell/m=1.6\times10^{4}$; (b) Construction~2 from~\cite[Sec 6.1]{zhang2016on},
with $\protect\ny=\protect\nx=10^{3}$, $\protect\nz=500$, $\ell=8$,
$m=0.125$, $\sqrt{m\ell}=1$, and $\kappa=64$.}
\end{figure}
When applied to a least-squares problem, ADMM reduces to a classic
block Gauss-Seidel splitting on the corresponding KKT equations, i.e.
the original saddle-point problem in (\ref{eq:basic_KKT}). Viewing
ADMM as the resulting linear fixed-point iterations, convergence can
be \emph{optimally} accelerated by using a Krylov subspace method
like generalized minimum residual (GMRES)~\cite{saad1986gmres,saad2003iterative}.
Or equivalently, viewing ADMM as a \emph{preconditioner}, it may be
used to improve the conditioning of the KKT equations for a Krylov
subspace method like GMRES. We refer to the GMRES-accelerated version
of ADMM (or the ADMM-preconditioned GMRES) as ADMM-GMRES.

In this paper, we show, using theoretical bounds and empirical results,
that ADMM-GMRES (nearly) achieves the optimal convergence rate in
(\ref{eq:optim_iter_cnt}) for every parameter choice. Figure~\ref{fig:iter_comparison}
makes this comparison for two representative problems. Our first main
result (Theorem~\ref{thm:real_eig}) conclusively establishes the
optimal iteration bound when $\beta$ is very large or very small.
Our second main result (Theorem~\ref{thm:cmp_eig}) proves a slightly
weaker statement: ADMM-GMRES converges within $O(\kappa^{2/3}\log\epsilon^{-1})$
iterations for all remaining choices of $\beta$, subject to a certain
normality assumption. The two bounds gives us the confidence to select
the parameter choice $\beta$ in order to maximize numerical stability. 

To validate these results, we benchmark the performance of ADMM-GMRES
with a randomly selected $\beta$ in Section~\ref{sec:AppEx} against
regular ADMM with an optimally selected $\beta$. Two problem classes
are considered: (1) random problems generated by selecting random
orthonormal bases and singular values; and (2) the Newton direction
subproblems associated with the interior-point solution of large-scale
semidefinite programs. Our numerical results suggest that ADMM-GMRES
converges in $O(\sqrt{\kappa}\log\epsilon^{-1})$ iterations for all
values of $\beta$, which is a slightly stronger iteration bound that
the one we have proved.

\subsection{Related ideas}

When the optimal parameter choice $\beta$ for (regular) ADMM is explicitly
available, we showed in a previous paper~\cite{zhang2016on} that
ADMM-GMRES can consistently converge in just $O(\kappa^{1/4}\log\epsilon^{-1})$
iterations, which is an entire order of magnitude better than the
optimal bound (\ref{eq:optim_iter_cnt}). Problems that would otherwise
require thousands of iterations to solve using ADMM are reduced to
just tens of ADMM-GMRES iterations. However, the improved rate cannot
be guaranteed over all problems, and there exist problem classes where
ADMM-GMRES convergences in $\Omega(\sqrt{\kappa}\log\epsilon^{-1})$
iterations for all choices of $\beta$.

More generally, the idea of using a preconditioned Krylov subspace
method to solve a saddle-point system has been explored in-depth by
a number of previous authors~\cite{battermann1998preconditioners,toh2004solving,oliveira2005new,bai2003hermitian,benzi2004preconditioner,benzi2005numerical}.
We make special mention of the Hemitian / Skew-Hermitian (HSS) splitting
method, first proposed by Bai, Golub \& Ng~\cite{bai2003hermitian}
and used as a preconditioner for saddle-point problems by Benzi \&
Golub~\cite{benzi2004preconditioner}, which also makes use of efficient
solutions to the subproblems in (\ref{eq:oracles}). It is curious
to note that HSS has a strikingly similar expression for its optimal
parameter choice and the resulting convergence rate, suggesting that
the two methods may be closely related.

Note that ADMM is an entirely distinct approach from the augmented
Lagrangian / method of multipliers (MM) in optimization, or equivalently,
the Uzawa method in saddle-point problems~\cite{elman1994inexact,bramble1997analysis}.
In MM, convergence is guaranteed in a small, constant number of iterations,
but each step requires the solution of an ill-conditioned symmetric
positive definite system of equations, often via preconditioned conjugate
gradients. In ADMM, convergence is slowly achieved over a large number
of iterations, but each iteration is relatively inexpensive. We refer
the reader to~\cite{boyd2011distributed} for a more detailed comparison
of the two methods.

Finally, it remains unknown whether these benefits extend to nonlinear
saddle-point problems (or equivalently, nonlinear versions of the
consensus problem), where the ADMM update equations are also nonlinear.
There are a number of competiting approaches to generalize GMRES to
nonlinear fixed-point iterations~\cite{saad1993flexible,brown1994convergence,fang2009two}.
Their application to ADMM is the subject of future work.

\subsection{Definitions \& Notation}

Given a matrix $M$, we use $\lambda_{i}(M)$ to refer to its $i$-th
eigenvalue, and $\Lambda\{M\}$ to denote its set of eigenvalues,
including multiplicities. If the eigenvalues are purely-real, then
$\lambda_{\max}(M)$ refers to its most positive eigenvalue, and $\lambda_{\min}(M)$
its most negative eigenvalue. Let $\|\cdot\|$ denote the $l_{2}$
vector norm, as well as the associated induced norm, also known as
the spectral norm. We use $\sigma_{i}(M)$ to refer to the $i$-th
largest singular value.

Define $m=\lambda_{\min}(\tilde{D})$ and $\ell=\lambda_{\max}(\tilde{D})$
as the strong convexity parameter and the gradient Lipschitz constant
for the quadratic form associated with the matrix $\tilde{D}=(AD^{-1}A^{T})^{-1}$.
The quantity $\kappa=\ell/m$ is the corresponding condition number.

\section{\label{sec:Preliminaries}Application of ADMM to the Saddle-point
Problem}

Beginning with a choice of the quadratic-penalty / step-size parameter
$\beta>0$ and initial points $\{x^{(0)},z^{(0)},y^{(0)}\}$, the
method generates iterates
\begin{align*}
\text{Local var. update: }x^{(k+1)} & =\arg\min_{x}\frac{1}{2}x^{T}Dx-r_{x}^{T}x+\frac{\beta}{2}\|Ax+Bz^{(k)}-c+\frac{1}{\beta}y^{(k)}\|^{2},\\
\text{Global var. update: }z^{(k+1)} & =\arg\min_{z}-r_{z}^{T}z+\frac{\beta}{2}\|Ax^{(k+1)}+Bz-c+\frac{1}{\beta}y^{(k)}\|^{2},\\
\text{Multiplier update: }y^{(k+1)} & =y^{(k)}+\beta(Ax^{(k+1)}+Bz^{(k+1)}-c).
\end{align*}
Note that the local and global variable updates can each be implemented
by calling one of the two subproblems in (\ref{eq:oracles}). Since
the KKT optimality conditions are linear with respect to the decision
variables, the update equations are also linear, and can be written
\begin{equation}
u^{(k+1)}=G_{\AD}(\beta)u^{(k)}+b(\beta),\label{eq:admm_update}
\end{equation}
with iteration matrix
\begin{equation}
G_{\AD}(\beta)=\begin{bmatrix}D+\beta A^{T}A & 0 & 0\\
\beta B^{T}A & \beta B^{T}B & 0\\
A & B & -\frac{1}{\beta}I
\end{bmatrix}^{-1}\begin{bmatrix}0 & -\beta A^{T}B & -A^{T}\\
0 & 0 & -B^{T}\\
0 & 0 & -\frac{1}{\beta}I
\end{bmatrix},\label{eq:ADMM_Gain}
\end{equation}
upon the vector of local, global, and multiplier variables, $u^{(k)}=[x^{(k)};z^{(k)};y^{(k)}]$.
We will refer to the residual norm in all discussions relating to
convergence.
\begin{defn}[Residual convergence]
\label{def:res_conv}Given the initial and final iterates $u^{(0)}=[x^{(0)};z^{(0)};y^{(0)}]$
and $u^{(k)}=[x^{(k)},z^{(k)},y^{(k)}]$, we say $\epsilon$ residual
convergence is achieved in $k$ iterations if $\|Mu^{(k)}-r\|\le\epsilon\|Mu^{(k)}-r\|$,
where $M$ and $r$ are the KKT matrix and vector in (\ref{eq:basic_KKT}).
\end{defn}

\subsection{Basic Spectral Properties}

Convergence analysis for linear fixed-point iterations is normally
performed by examining the spectral properties of the corresponding
iteration matrix. Using dual feasibility arguments, a block-Schur
decomposition for (\ref{eq:ADMM_Gain}) can be explicitly specified.
\begin{lem}[{\cite[ Lem. 11]{zhang2016on}}]
\label{lem:blkschur_ad}Define the QR decomposition $B=QR$ with
$Q\in\R^{\ny\times\nz}$ and $R\in\R^{\nz\times\nz}$, and define
$P\in\R^{p\times(\ny-\nz)}$ as its orthogonal complement. Then defining
the orthogonal matrix $U$ and the scaling matrix $S(\beta)$,
\begin{equation}
U=\left[\begin{array}{c|cc|c}
I_{\nx} & 0 & 0 & 0\\
0 & I_{\nz} & 0 & 0\\
0 & 0 & P & Q
\end{array}\right],\qquad S(\beta)=\left[\begin{array}{c|cc|c}
\beta I_{\nx} & 0 & 0 & 0\\
\hline 0 & \beta R & 0 & 0\\
0 & 0 & I_{\ny-\nz} & 0\\
\hline 0 & 0 & 0 & I_{\nz}
\end{array}\right]\label{eq:schur_basis_mm}
\end{equation}
yields a block-Schur decomposition of $G_{\AD}(\beta)$
\begin{equation}
U^{T}G_{\AD}(\beta)U=S^{-1}(\beta)\left[\begin{array}{c|c|c}
0_{\nx} & G_{12}(\beta) & G_{13}(\beta)\\
\hline 0 & G_{22}(\beta) & G_{23}(\beta)\\
\hline 0 & 0 & 0_{\nz}
\end{array}\right]S(\beta),\label{eq:AD_schur_blks}
\end{equation}
where the size $\ny\times\ny$ inner iteration matrix $G_{22}(\beta)=\frac{1}{2}I+\frac{1}{2}K(\beta)$
is defined in terms of the matrix
\begin{equation}
K(\beta)=\begin{bmatrix}Q^{T}\\
-P^{T}
\end{bmatrix}[(\beta^{-1}\tilde{D}+I)^{-1}-(\beta\tilde{D}^{-1}+I)^{-1}]\begin{bmatrix}Q & P\end{bmatrix},\label{eq:Kdef}
\end{equation}
and $\tilde{D}=(AD^{-1}A^{T})^{-1}$.
\end{lem}
We may immediately conclude that $G_{\AD}(\beta)$ has $\nx+\nz$
zero eigenvalues, and $\ny$ nonzero eigenvalues the lie within a
disk on the complex plane centered at $+\frac{1}{2}$, with radius
of $\frac{1}{2}\|K(\beta)\|$. It is straightforward to compute the
radius of this disk exactly.
\begin{lem}
\label{lem:Knrm}Let $\tilde{D}=(AD^{-1}A^{T})^{-1}$, and define
$m=\lambda_{\min}(\tilde{D})$ and $\ell=\lambda_{\max}(\tilde{D})$.
Then the spectral norm of $K(\beta)$ is given
\begin{equation}
\|K(\beta)\|=\frac{\gamma-1}{\gamma+1},\text{ where }\gamma=\max\left\{ \frac{\beta}{m},\frac{\ell}{\beta}\right\} .\label{eq:ad_sprd_bnd_a}
\end{equation}

\end{lem}
Also, we see from (\ref{eq:AD_schur_blks}) that each Jordan block
associated with a zero eigenvalue of $G_{\AD}(\beta)$ is at most
size $2\times2$. After two iterations, the behavior of ADMM becomes
entirely dependent upon the inner iteration matrix $G_{22}(\beta)=\frac{1}{2}I+\frac{1}{2}K(\beta)$.
\begin{lem}[{\cite[ Lem. 13]{zhang2016on}}]
\label{lem:lem_poly_2norm}For any $\beta$ and any polynomial $p(\cdot)$,
we have
\[
\|p(G_{\AD}(\beta))\, G_{\AD}^{2}(\beta)\|\le c_{1}(\beta)\|p(G_{22}(\beta))\|,
\]
where $c_{1}(\beta)$ is defined in terms of the matrices in Lemma~\ref{lem:blkschur_ad},
as in 
\[
c_{1}(\beta)=\|S(\beta)\|\|S^{-1}(\beta)\|\|G_{\AD}(\beta)\|^{2}.
\]

\end{lem}
One application of Lemma~\ref{lem:lem_poly_2norm} is to bound the
spectral norm of the $k$-th power iteration, i.e. $\|G_{\AD}^{k}(\beta)\|$,
thereby yielding the following iteration estimate.
\begin{prop}[{\cite[Prop. 7]{zhang2016on}}]
\label{prop:admm_conv}ADMM with fixed parameter $\beta=\sqrt{m\ell}$
attains $\epsilon$ residual convergence in 
\[
2+\left\lceil (\kappa^{\frac{1}{2}}+1)\log(c_{1}\kappa_{M}\epsilon^{-1})\right\rceil \text{ iterations,}
\]
where $c_{1}$ is defined in Lemma~\ref{lem:lem_poly_2norm}, and
$\kappa_{M}=\|M\|\|M^{-1}\|$ with $M$ defined in (\ref{eq:basic_KKT}). 
\end{prop}

\subsection{\label{sub:conv_analy_GMRES}Accelerating Convergence using GMRES}

In the context of quadratic objectives, the convergence of ADMM can
be accelerated by GMRES in a largely plug-and-play manner. Given a
specific choice of parameter $\beta>0$ and an initial point $u^{(0)}=[x^{(0)};z^{(0)};y^{(0)}]$,
we may task GMRES with the fixed-point equation associated with the
ADMM update equation (\ref{eq:admm_update})
\begin{equation}
u^{\star}-G_{\AD}(\beta)u^{\star}=b(\beta),\label{eq:admm_fixed}
\end{equation}
which is indeed a linear system of equations when $\beta$ is held
fixed. It is an elementary fact that the resulting iterates will always
converge onto the fixed-point point faster than regular ADMM (under
a suitably defined metric)~\cite{saad1986gmres}.

Alternatively, the fixed-point equation (\ref{eq:admm_fixed}) is
equivalent to the \emph{left-preconditioned} system of equations
\begin{equation}
P_{\AD}^{-1}(\beta)[Mu^{\star}-r]=0\qquad\Leftrightarrow\qquad\text{(\ref{eq:admm_fixed})},\label{eq:left-prec}
\end{equation}
where $M$ and $r$ are the KKT matrix and residual defined in (\ref{eq:basic_KKT}),
the \emph{ADMM preconditioner matrix} is
\begin{equation}
P_{\AD}(\beta)=\begin{bmatrix}I & 0 & -\beta A^{T}\\
0 & I & -\beta B^{T}\\
0 & 0 & I
\end{bmatrix}\begin{bmatrix}D+\beta A^{T}A & 0 & 0\\
\beta B^{T}A & \beta B^{T}B & 0\\
A & B & -\frac{1}{\beta}I
\end{bmatrix}=\begin{bmatrix}D & -\beta A^{T}B & A^{T}\\
0 & 0 & B^{T}\\
A & B & -\frac{1}{\beta}I
\end{bmatrix}.\label{eq:Pad_def}
\end{equation}
Note that the ADMM iteration matrix satisfies $G_{\AD}(\beta)=I-P_{\AD}^{-1}(\beta)M$
by definition. In turn, GMRES-accelerated ADMM is equivalent to a
preconditioned GMRES solution to the KKT system, $Mu=r$, with preconditioner
$P_{\AD}(\beta)$. Matrix-vector products with $P_{\AD}^{-1}(\beta)$
can always be implemented as the composition of an augmentation operation
and a single iteration of ADMM, as seen in the factorization in (\ref{eq:Pad_def}).

GMRES can also be used to solve the \emph{right-preconditioned} system
\begin{equation}
MP_{\AD}^{-1}(\beta)\hat{u}-r=0,\label{eq:right-prec}
\end{equation}
and the solution is recovered via $u=P_{\AD}^{-1}(\beta)\hat{u}$.
The resulting method performs essentially the same steps as the one
above, but optimizes the iterates under a more preferable metric.
Starting from the same initial point $u^{(0)}$, the $k$-th iterate
of GMRES as applied to (\ref{eq:right-prec}), written $u_{\GM}^{(k)}$,
is guaranteed to produce a KKT residual norm that is smaller than
or equal to that of the $k$-th iterate of regular ADMM, written $u_{\AD}^{(k)}$,
as in $\|Mu_{\GM}^{(k)}-r\|\le\|Mu_{\AD}^{(k)}-r\|$. This property
is preferable as $P_{\AD}(\beta)$ becomes progressively ill-conditioned
and numerical precision becomes an issue; cf.~\cite[Sec. 7]{zhang2016on}
for details.

Throughout this paper, we will refer to both methods as GMRES-accelerated
ADMM, or ADMM-GMRES for short, and reserve the ``left-preconditioned''
or the ``right-preconditioned'' specifications only where the distinctions
are important. The reason is that both methods share a common bound
for the purposes of convergence analysis.
\begin{prop}
\label{prop:admm_gmres}Given fixed $\beta>0$, let $u^{(k)}$ be
the iterate generated at the $k$-th iteration of GMRES as applied
to either (\ref{eq:left-prec}) or (\ref{eq:right-prec}). Then the
following bounds hold for all $k\ge2$ 
\[
\frac{\|Mu^{(k)}-r\|}{\|Mu^{(0)}-r\|}\le c_{1}\kappa_{P}\min_{\begin{subarray}{c}
p\in\P_{k-2}\\
p(1)=1
\end{subarray}}\|p(K)\|,
\]
where $K\equiv K(\beta)$ is defined in (\ref{eq:Kdef}), $c_{1}$
is defined in Lemma~\ref{lem:lem_poly_2norm}, $\kappa_{P}=\|P_{\AD}\|\|P_{\AD}^{-1}\|$
with $P_{\AD}\equiv P_{\AD}(\beta)$ defined in (\ref{eq:Pad_def}),
and $\P_{k}$ denotes the space of order-$k$ polynomials.\end{prop}
\begin{proof}
Given an arbitrary linear system, $Au=b$, GMRES generates iterates
$u^{(k)}$ that satisfies the minimal residual property~\cite{saad1986gmres}
\begin{equation}
\|r^{(k)}\|/\|r^{(0)}\|\le\min_{\begin{subarray}{c}
p\in\P_{k}\\
p(1)=1
\end{subarray}}\|p(I-A)\|,\label{eq:gmres_bnd}
\end{equation}
where $r^{(k)}=Au^{(k)}-b$ is the $k$-th residual vector. Furthermore,
Lemma~\ref{lem:lem_poly_2norm} yields for all $k\ge2$, 
\begin{equation}
\min_{\begin{subarray}{c}
p\in\P_{k}\\
p(1)=1
\end{subarray}}\|p(G_{\AD})\|\le c_{1}\min_{\begin{subarray}{c}
p\in\P_{k-2}\\
p(1)=1
\end{subarray}}\|p(G_{22})\|=c_{1}\min_{\begin{subarray}{c}
p\in\P_{k-2}\\
p(1)=1
\end{subarray}}\|p(K)\|,\label{eq:admm_gmres3}
\end{equation}
and the last equality is due to the existence of a bijective linear
map between $G_{22}\cup\{1\}$ and $K\cup\{1\}$. In the left-preconditioned
system (\ref{eq:left-prec}), the matrix $I-A$ is $I-P_{\AD}^{-1}M=G_{\AD}$,
and the $\kappa_{P}$ factor arises by bounding the residuals $\|P_{\AD}r^{(k)}\|/\|P_{\AD}r^{(0)}\|\le\kappa_{P}\|r^{(k)}\|/\|r^{(0)}\|$.
In the right-preconditioned system (\ref{eq:left-prec}), the matrix
$I-A$ is $I-MP_{\AD}^{-1}=P_{\AD}G_{\AD}P_{\AD}^{-1}$, and the $\kappa_{P}$
factor arises via $\|p(P_{\AD}G_{\AD}P_{\AD}^{-1})\|\le\kappa_{P}\|p(G_{\AD})\|$.
\end{proof}
In order to use Proposition~\ref{prop:admm_gmres} to derive useful
convergence estimates, the polynomial norm-minimization problem can
be reduced into a polynomial min-max approximation problem over a
set of points on the complex plane. More specifically, consider the
following sequence of inequalities
\begin{equation}
\min_{\begin{subarray}{c}
p\in\P_{k}\\
p(1)=1
\end{subarray}}\|p(K)\|\le\min_{\begin{subarray}{c}
p\in\P_{k}\\
p(1)=1
\end{subarray}}\|Xp(\Lambda)X^{-1}\|\le\kappa_{X}\min_{\begin{subarray}{c}
p\in\P_{k}\\
p(1)=1
\end{subarray}}\max_{\lambda\in\Lambda\{K\}}|p(\lambda)|,\label{eq:eig_prob}
\end{equation}
which makes the following normality assumption, that is standard within
this context.

\begin{assume}[$\kappa_X$ is bounded]\label{ass:kappaX}Given fixed
$\beta>0$, the matrix $K\equiv K(\beta)$, defined in (\ref{eq:Kdef}),
is diagonalizable with eigendecomposition, $K=X\Lambda X^{-1}$. Furthermore,
the condition number for the matrix-of-eigenvectors, $\kappa_{X}=\|X\|\|X^{-1}\|$,
is bounded from above by an absolute constant.\end{assume}

We refer to this last problem in (\ref{eq:eig_prob}) as the \emph{eigenvalue
approximation problem}. Only in very rare cases is an explicit closed-form
solution known, but any heuristic choice of polynomial $p(\cdot)$
will provide a valid upper-bound.

\section{Convergence Analysis for ADMM-GMRES}

Our main result in this paper is that ADMM-GMRES converges to an $\epsilon$-accurate
solution in $O(\kappa^{2/3}\log\epsilon^{-1})$ iterations for any
value of $\beta>0$, in the sense of the residual. We split the precise
statement into two parts. First, for very large and very small values
of $\beta$, we can conclusively establish that ADMM-GMRES convergences
in $O(\sqrt{\kappa}\log\epsilon^{-1})$ iterations. This is asymptotically
the same as the optimal figure for regular ADMM.
\begin{thm}[Extremal $\beta$]
\label{thm:real_eig}For any choice of $\beta>\ell$ or $0<\beta<m$,
GMRES-accelerated ADMM generates the iterate $u^{(k)}=[x^{(k)};z^{(k)};y^{(k)}]$
at the $k$-th iteration that satisfies
\[
\frac{\|Mu^{(k)}-r\|}{\|Mu^{(0)}-r\|}\le2\, c_{1}\,\kappa_{P}\left[1+\left(\max\left\{ \frac{\beta}{\ell},\frac{m}{\beta}\right\} -1\right)^{-1}\right]\left(\frac{\sqrt{2\kappa}-1}{\sqrt{2\kappa}+1}\right)^{0.317\, k}
\]
where $\kappa=\ell/m$ and the factors $c_{1},\kappa_{P}$ are polynomial
in $\beta+\beta^{-1}$ and defined in Lemma~\ref{lem:lem_poly_2norm}
and Proposition~\ref{prop:admm_gmres}.\end{thm}
\begin{cor}
GMRES-accelerated ADMM achieves $\epsilon$ residual convergence in
\[
O(\sqrt{\kappa}\log\epsilon^{-1}+\sqrt{\kappa}|\log\beta|)\text{ iterations}
\]
for any choice of $\beta>\ell$ or $0<\beta<m$.
\end{cor}
For intermediate choices of $\beta$, the same result \emph{almost
holds}. Subject to the normality assumption in Assumption~\ref{ass:kappaX},
ADMM-GMRES converges in $O(\kappa^{2/3}\log\epsilon^{-1})$ iterations.
Accordingly, we conclude that ADMM-GMRES converge within this number
of iterations for every fixed value of $\beta>0$.
\begin{thm}[Intermediate $\beta$]
\label{thm:cmp_eig}For any choice of $m\le\beta\le\ell$, GMRES-accelerated
ADMM generates the iterate $u^{(k)}=[x^{(k)};z^{(k)};y^{(k)}]$ at
the $k$-th iteration that satisfies
\[
\frac{\|Mu^{(k)}-r\|}{\|Mu^{(0)}-r\|}\le2\, c_{1}\,\kappa_{P}\,\kappa_{X}\left(\frac{\kappa^{2/3}}{\kappa^{2/3}+1}\right)^{0.209\, k}
\]
where $\kappa=\ell/m$, the factors $c_{1},\kappa_{P}$ are polynomial
in $\beta$ and defined in Lemma~\ref{lem:lem_poly_2norm} and Proposition~\ref{prop:admm_gmres},
and the factor $\kappa_{X}$ is defined in Assumption~\ref{ass:kappaX}.\end{thm}
\begin{cor}
\label{cor:conv_23}GMRES-accelerated ADMM achieves $\epsilon$ residual
convergence in
\[
O(\kappa^{2/3}\log\epsilon^{-1}+\kappa^{2/3}|\log\beta|)\text{ iterations}
\]
for any choice of $\beta>0$.
\end{cor}
As we reviewed in Section~\ref{sec:Preliminaries}, convergence analysis
for GMRES can be reduced to a polynomial approximation problem over
the eigenvalues of $K(\beta)$, written in (\ref{eq:eig_prob}). Our
proofs for Theorems~\ref{thm:real_eig}~\&~\ref{thm:cmp_eig} are
based on solving this polynomial approximation problem heuristically.
The discrete eigenvalue distribution of $K(\beta)$ is enclosed within
simple regions on the complex plane in Section~\ref{sec:Eigenvalue-Distribution}
for different values of $\beta$. The polynomial approximation problems
associated with these outer enclosures are solved in their most general
form in Section~\ref{sec:approx_probs}. These results are pieced
together in Section~\ref{sec:Proof-main}, yielding proofs to our
main results.

\section{\label{sec:Eigenvalue-Distribution}Eigenvalue Distribution of the
Iteration Matrix}

\begin{figure}
\subfloat[]{\includegraphics[width=0.3\columnwidth]{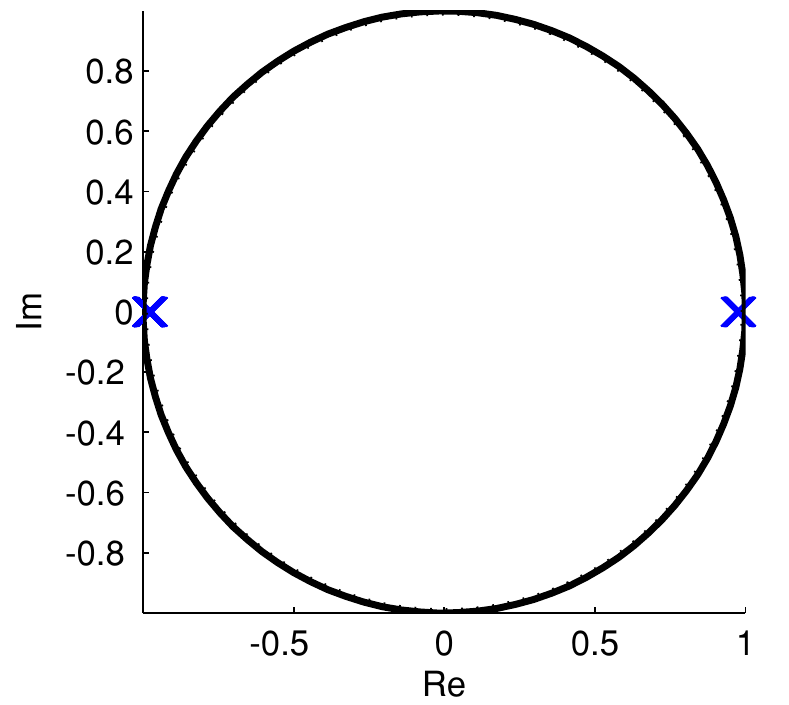}

}\subfloat[]{\includegraphics[width=0.3\columnwidth]{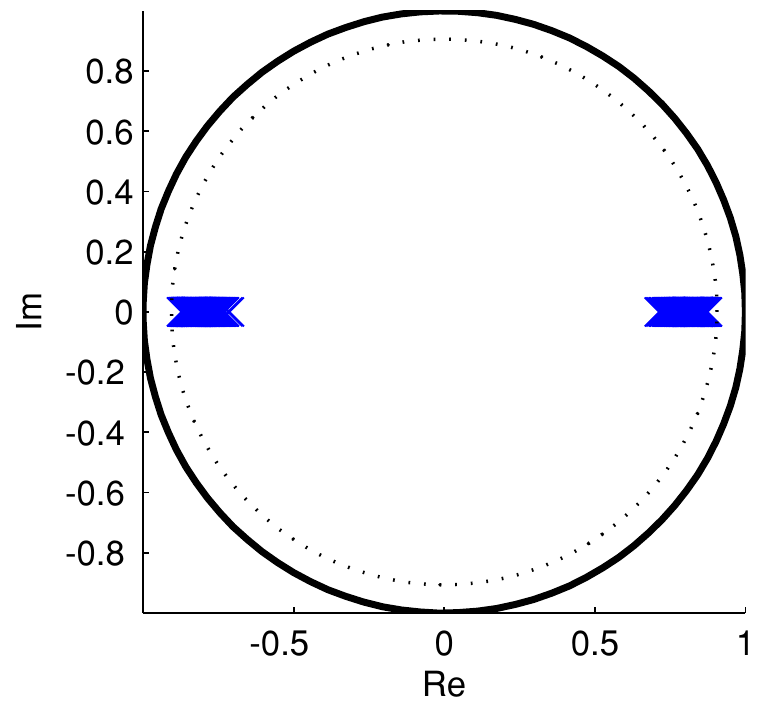}

}\subfloat[]{\includegraphics[width=0.3\columnwidth]{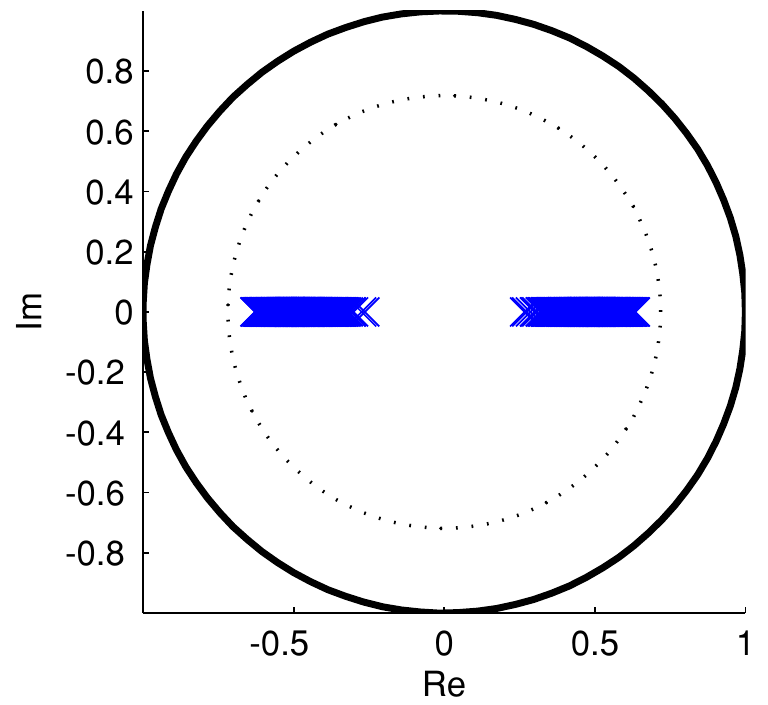}

}

\subfloat[]{\includegraphics[width=0.3\columnwidth]{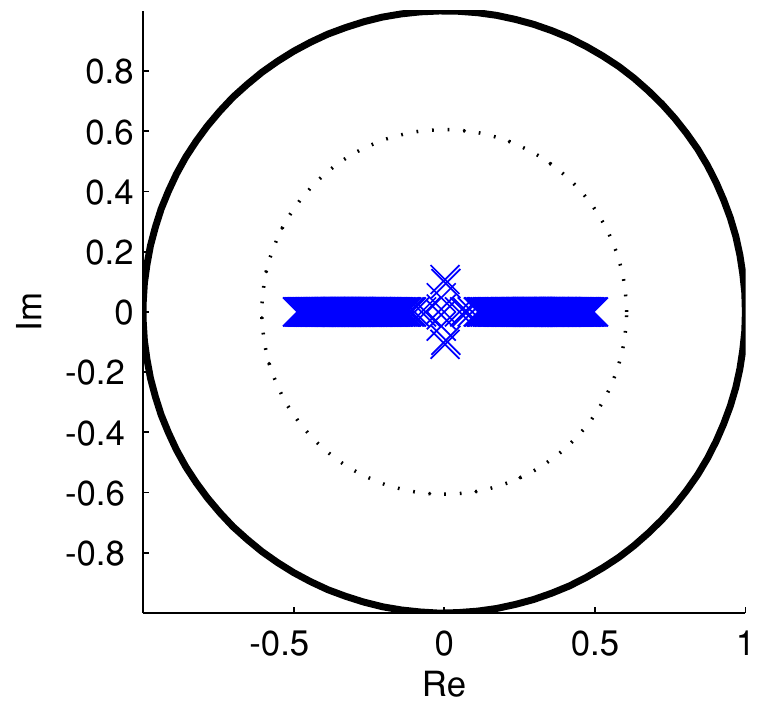}}\subfloat[]{\includegraphics[width=0.3\columnwidth]{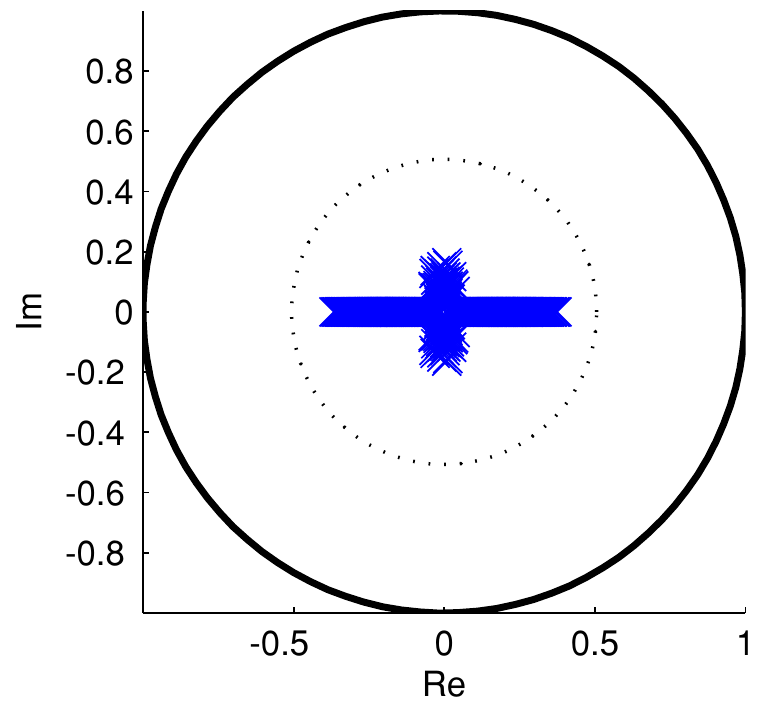}}\subfloat[]{\includegraphics[width=0.3\columnwidth]{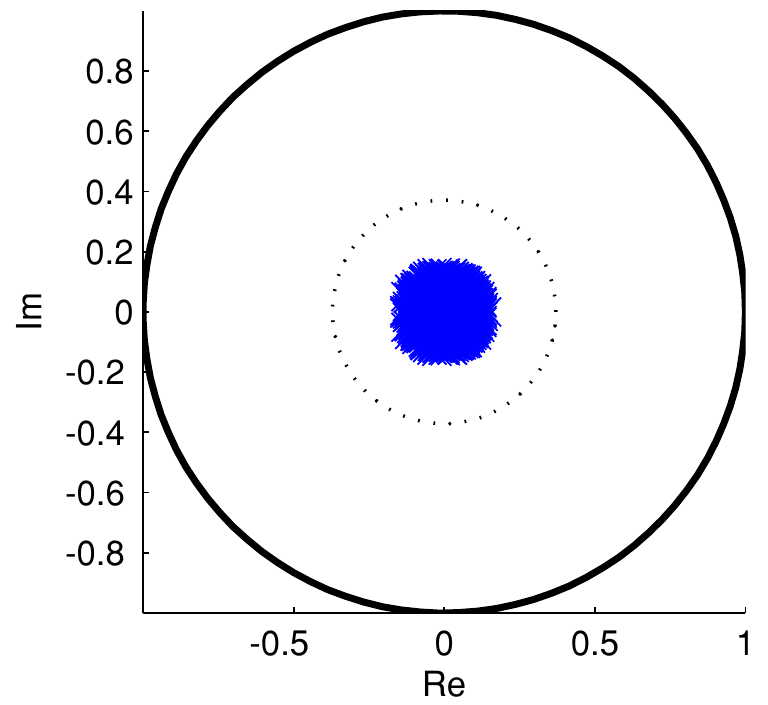}}

\protect\caption{\label{fig:ev_sweep}Eigenvalues (markers) and the spectral norm (dotted
circle) of $K(\beta)$ for a randomly generated problem with $\protect\ny=\protect\nx=1000$,
$\protect\nz=500$, $m=0.49$ and $\ell=2.2$: (a) $\beta=0.01$;
(b) $\beta=0.1$; (c) $\beta=0.33$; (d) $\beta=0.5$; (e) $\beta=0.67$;
(f) $\beta=1$. The unit circle is shown in as a solid circle.}
\end{figure}
The eigenvalues of the iteration matrix $K(\beta)$ play a pivotal
role in driving the convergence of both ADMM as well as ADMM-GMRES.
Figure~\ref{fig:ev_sweep} plots these for a fixed, randomly generated
problem, while sweeping the value of $\beta$. Initially, we see two
clusters of purely-real eigenvalues, tightly concentrated about $\pm1$,
that enlargen and shift closer towards the origin and towards each
other with increasing $\beta$. As the two clusters coalesce, some
of the purely-real eigenvalues become complex. The combined radius
of the two clusters reaches its minimum at around $\beta=\sqrt{m\ell}$,
at which point most of the eigenvalues are complex. Although not shown,
the process is reversed once $\beta$ moves past $\sqrt{m\ell}$;
the two clusters shrink, become purely-real, break apart, and move
away from the origin, ultimately reverting into two clusters concentrated
about $\pm1$.

Three concrete findings can be summarized from these observations.
First, despite the fact that $K(\beta)$ is nonsymmetric, its eigenvalues
are purely-real over a broad range of $\beta$. 
\begin{lem}
\label{lem:eig_real}Let $\beta>\ell$ or $\beta<m$. Then $K(\beta)$
is diagonalizable and its eigenvalues are purely real. Furthermore,
let $\kappa_{X}$ be the condition number for the matrix-of-eigenvectors
as defined in Assumption~\ref{ass:kappaX}. Then this quantity is
bound 
\begin{equation}
\kappa_{X}\le1+\left(\max\left\{ \frac{\beta}{\ell},\frac{m}{\beta}\right\} -1\right)^{-1}.\label{eq:condest}
\end{equation}

\end{lem}
Furthermore, the eigenvalues are partitioned into two distinct, purely-real
clusters that only become complex once they coalesce. 
\begin{lem}
\label{lem:eig_distr}Define the positive scalar $\gamma=\max\{\beta/m,\ell/\beta\}$,
which satisfies $\gamma\ge\sqrt{\kappa}$ by construction. If $\gamma\in[\sqrt{\kappa},\kappa]$,
then $\Lambda(K)$ is enclosed within the union of a disk and an interval:
\begin{equation}
\Lambda(K)\subset\left\{ z\in\C:|z|\le\frac{\kappa}{\gamma+\kappa}-\frac{1}{\gamma+1}\right\} \cup\left[-\frac{\gamma-1}{\gamma+1},+\frac{\gamma-1}{\gamma+1}\right].\label{eq:eig_case1}
\end{equation}
If $\gamma\in(\kappa,2\kappa]$, then $\Lambda(K)$ is enclosed within
a single interval:
\begin{equation}
\Lambda(K)\subset\left[-\frac{\gamma-1}{\gamma+1},+\frac{\gamma-1}{\gamma+1}\right].\label{eq:eig_case2}
\end{equation}
Finally, if $\gamma\in(2\kappa,\infty)$, then $\Lambda(K)$ is enclosed
within the union of two disjoint intervals:
\begin{equation}
\Lambda(K)\subset\left[-\frac{\gamma-1}{\gamma+1},-\frac{\gamma-2\kappa}{\gamma+\kappa}\right]\cup\left[+\frac{\gamma-2\kappa}{\gamma+\kappa},+\frac{\gamma-1}{\gamma+1}\right].\label{eq:eig_case3}
\end{equation}
Furthermore, if $0<\nx<\ny$, then $\Lambda(K)$ contains at least
one eigenvalue within each interval.
\end{lem}
In the limits $\beta\to0$ and $\beta\to\infty$, the two clusters
in (\ref{eq:eig_case3}) concentrate about $\pm1$, and the spectral
radius of the ADMM iteration matrix converges towards~1.
\begin{cor}
\label{cor:sprad_lb}Let $\beta>2\ell$ or $\beta<\frac{1}{2}m$.
Then for $0<\nz<\ny$, there exists an eigenvalue of the ADMM iteration
matrix $\lambda_{i}\in\Lambda\{G_{\AD}(\beta)\}$ whose modulus is
lower-bounded 
\[
|\lambda_{i}|\ge\frac{\gamma-\kappa}{\gamma+\kappa},\text{ where }\gamma=\max\left\{ \frac{\beta}{m},\frac{\ell}{\beta}\right\} .
\]

\end{cor}
The spectral radius determines the asympotic convergence rate, so
given Corollary~\ref{cor:sprad_lb}, it is unsurprising that ADMM
stagnates if $\beta$ is poorly chosen. But the situation is different
with ADMM-GMRES, because it is able to exploit the clustering of eigenvalues.
As we will see later, this is the mechanism that allows ADMM-GMRES
to be insensitive to the parameter choice.

\subsection{Properties of $J$-symmetric matrices}

Most of our characterizations for the eigenvalues of $K(\beta)$ are
based on a property known as ``$J$-symmetry''. In the following
discussion, we will drop all arguments with respect to $\beta$ for
clarity. Returning to its definition in (\ref{eq:Kdef}), we note
that $K$ has the block structure
\begin{equation}
K=\begin{bmatrix}X & Z\\
-Z^{T} & Y
\end{bmatrix},\label{eq:Kblks}
\end{equation}
with subblocks $X\in\R^{\nz\times\nz}$, $Y\in\R^{(\ny-\nz)\times(\ny-\nz)}$,
and $Z\in\R^{\nz\times(\ny-\nz)}$, 
\begin{gather}
X=Q^{T}\tilde{K}Q,\qquad Y=-P^{T}\tilde{K}P,\qquad Z=Q^{T}\tilde{K}P,\\
\tilde{K}=(\beta^{-1}\tilde{D}+I)^{-1}-(\beta\tilde{D}^{-1}+I)^{-1},\label{eq:Ktilde_def}
\end{gather}
and the matrices $Q$ and $P$ with orthonormal columns are defined
as in Lemma~\ref{lem:blkschur_ad}. From the block structure in (\ref{eq:Kblks})
we see that the matrix $K$ is self-adjoint with respect to the indefinite
product (assuming $0<\nz<\ny$) defined by $J=\blkdiag(I_{\nz},-I_{(\ny-\nz)})$:
\[
\left\langle y,Mx\right\rangle _{J}=\left\langle My,x\right\rangle _{J}\qquad\iff\qquad y^{T}JMx=(My)^{T}Jx.
\]
Matrices that have this property frequently appear in saddle-point
type problems; cf.~\cite{benzi2005numerical,benzi2006eigenvalues}
for a more detailed treatment of this subject. Much can be said about
their spectral properties.
\begin{prop}
\label{prop:Jsym_comp_radius}The $J$-symmetric matrix $K$ in (\ref{eq:Kblks})
has at most $2\min\{\nz,\ny-\nz\}$ eigenvalues with nonzero imaginary
parts, counting conjugates. These eigenvalues are contained within
the disk $\mathcal{D}_{a}=\{z\in\C:|z|\le a\}$ of radius
\[
a=\min_{\eta\in\R}\|K+\eta J\|,
\]
where $\P$ denotes the space of polynomials.\end{prop}
\begin{proof}
Benzi~\&~Simoncini~\cite{benzi2006eigenvalues} provide a succinct
proof for the first statement. The second statement is based on the
fact that every eigenpair $\{\lambda_{i},x_{i}\}$ of $K$ satisfying
$\mathrm{Im}(\lambda_{i})\ne0$ must have an eigenvector $x_{i}$
that is ``$J$-neutral'', i.e. satisfying $x_{i}^{*}Jx_{i}=0$;
cf.~\cite[Thm. 2.1]{benzi2006eigenvalues}. Hence, the following
bound holds 
\begin{equation}
|\lambda_{i}|\le\max_{\|x\|=1}\{|x^{*}Kx|:x^{*}Jx=0\}\label{eq:Jsym_fov}
\end{equation}
for every $\lambda_{i}$ with $\mathrm{Im}(\lambda_{i})\ne0$. Taking
the Lagrangian dual of (\ref{eq:Jsym_fov}) yields the desired statement.
\end{proof}
Also, we can derive a simple sufficient condition for the eigenvalues
of $K$ to be purely real, based on the ideas described in~\cite{benzi2006eigenvalues}. 
\begin{prop}
\label{prop:Jsym_real}Suppose that there exists a real scalar $\eta\ne0$
to make the matrix $H=\eta JK$ positive definite. Then $K$ is diagonalizable
with eigendecomposition, $K=X\Lambda X^{-1}$, its eigenvalues are
purely-real, and the condition number of the matrix-of-eigenvectors
satisfies $\kappa_{X}\triangleq\|X\|\|X^{-1}\|\le\sqrt{\|H\|\|H^{-1}\|}$\end{prop}
\begin{proof}
It is easy to verify that $K$ is also symmetric with respect to $H$,
as in $KM=K^{T}H$. Since $H$ is positive definite, there exists
a symmetric positive definite matrix $W=W^{T}$ satisfying $W^{2}=H$,
and the $H$-symmetry implies
\[
W(WMW^{-1})W=W(W^{-1}M^{T}W)W\quad\iff\quad WMW^{-1}=(WMW^{-1})^{T}=\tilde{M}.
\]
Hence we conclude that $M$ is similar to the real symmetric matrix
$\tilde{M}$, with purely-real eigenvalues and eigendecomposition
$\tilde{M}=V\Lambda V^{T}$, where $V$ is orthogonal. The corresponding
eigendecomposition for $M$ is $M=X\Lambda X^{-1}$ with $X=W^{-1}V$. 
\end{proof}
Finally, we may use the block-generalization of Gershgorin's circle
theorem to decide when the off-diagonal block $Z$ is sufficiently
``small'' such that the eigenvalues of $K$ become similar to the
block diagonal matrix $\blkdiag(X,Y)$.
\begin{prop}
\label{prop:Jsym_Gersh}Given $J$-symmetric matrix $K$ in (\ref{eq:Kblks}),
define the two Gershgorin sets
\[
\mathcal{G}_{X}=\bigcup_{i=1}^{n}\{z\in\C:|z-\lambda_{i}(X)|\le\|Z\|\},\qquad\mathcal{G}_{Y}=\bigcup_{i=1}^{m}\{z\in\C:|z-\lambda_{i}(Y)|\le\|Z\|\}.
\]
Then $\Lambda\{K\}\subset\mathcal{G}_{X}\cup\mathcal{G}_{Y}$. Moreover,
if $\mathcal{G}_{X}$ and $\mathcal{G}_{Y}$ are disjoint, i.e. $\mathcal{G}_{X}\cap\mathcal{G}_{Y}=\emptyset$,
then $\Lambda\{K\}$ contains exactly $\nz$ eigenvalues in $\mathcal{G}_{X}$
and $\ny-\nz$ eigenvalues in $\mathcal{G}_{Y}$.\end{prop}
\begin{proof}
This is a straightforward application of the block Gershgorin's theorem
for matrices with normal pivot blocks~\cite[ Thm. 4]{feingold1962block}.
\end{proof}

\subsection{Proof of Lemma~\ref{lem:eig_real}}
\begin{proof}
The cases of $\nz=0$ and $\nz=\ny$ are trivial. In the remaining
cases, $K$ is $J$-symmetric, and we will use Proposition~\ref{prop:Jsym_real}
to prove the statement. Noting that $JK$ is unitarily similar with
$\tilde{K}$, simple direct computation reveals that $+JK$ is positive
definite for $\beta>\ell$, and $-JK$ is positive definite for $\beta<m$.
Hence, for these choices of $\beta$, the eigenvalues of $K$ are
purely real. Some further computation reveals that $\|JK\|=(\gamma-1)/(\gamma+1)$
and $\|(JK)^{-1}\|=(\gamma+\kappa)/(\gamma-\kappa)$, so the condition
number of $JK$ is bound 
\begin{equation}
\|JK\|\|(JK)^{-1}\|<\|(JK)^{-1}\|=\frac{\gamma+\kappa}{\gamma-\kappa}=1+\frac{2}{\gamma/\kappa-1}.
\end{equation}
Taking the square-root and substituting $\sqrt{1+2x}\le1+x$ yields
the desired estimate for $\kappa_{X}$ in (\ref{eq:condest}).
\end{proof}

\subsection{Proof of Lemma~\ref{lem:eig_distr}}

\begin{proof}
Again, the cases of $\nz=0$ and $\nz=\ny$ are trivial. In all remaining
cases, $K$ is $J$-symmetric. The single interval case (\ref{eq:eig_case2})
is a trivial consequence of our purely-real result in Lemma~\ref{lem:eig_real}.
The disk-and-interval case (\ref{eq:eig_case1}) arises from Proposition~\ref{prop:Jsym_comp_radius},
in which we use the disk of radius
\[
a=\min_{\eta\in\R}\|K-\eta J\|=\min_{\eta\in\R}\|\tilde{K}-\eta I\|=\frac{\kappa}{\gamma+\kappa}-\frac{1}{\gamma+1}
\]
to enclose the eigenvalues with nonzero imaginary parts, and the spectral
norm disk $|\lambda_{i}(K)|\le\|K\|$ to enclose the purely-real eigenvalues.
The two-interval case (\ref{eq:eig_case3}) is a consequence of the
block Gershgorin theorem in Proposition~\ref{prop:Jsym_Gersh}. Some
direct computation on the block matrices in (\ref{eq:Kblks}) yields
the following two statements 
\begin{align}
\Lambda\{X\}\cup\Lambda\{Y\} & \subset\left[-\frac{\gamma-1}{\gamma+1},-\frac{\gamma-\kappa}{\gamma+\kappa}\right]\cup\left[+\frac{\gamma-\kappa}{\gamma+\kappa},+\frac{\gamma-1}{\gamma+1}\right]\subset\R,\label{eq:lem_clust_2}\\
\|Z\| & \le\frac{\kappa}{\gamma+\kappa},\label{eq:lem_clust_3}
\end{align}
the latter of which also uses the fact that $\|Q^{T}\tilde{K}P\|=\|Q^{T}(\tilde{K}-\eta I)P\|\le\|\tilde{K}-\eta I\|$.
The projections of the corresponding Gershgorin regions onto the real
line satisfy
\[
\mathrm{Re}\{\mathcal{G}_{-}\}\subset\left[-\frac{\gamma-1}{\gamma+1}-\frac{\kappa}{\gamma+\kappa},-\frac{\gamma-2\kappa}{\gamma+\kappa}\right],\quad\mathrm{Re}\{\mathcal{G}_{+}\}\subset\left[\frac{\gamma-2\kappa}{\gamma+\kappa},\frac{\gamma-1}{\gamma+1}+\frac{\kappa}{\gamma+\kappa}\right].
\]
Once $\gamma>2\kappa$, the regions become separated. Taking the intersection
between each disjoint Gershgorin region, the real line, and the spectral
norm disk $|\lambda_{i}(K)|\le\|K\|$ yields (\ref{eq:eig_case3}).
\end{proof}

\section{\label{sec:approx_probs}Solving the Approximation Problems}

With the distribution of eigenvalues characterized in Lemma~\ref{lem:eig_distr},
we now consider solving each of the three accompanying approximation
problems in their most general form.

\subsection{Chebyshev approximation over the real interval}

The optimal approximation for an interval over the real line has a
closed-form solution due to a classic result attributed to Chebyshev.
\begin{thm}
\label{thm:seg_estim}Let $\mathcal{I}$ denote the interval $[c-a,c+a]$
on the real line. Then assuming that $+1\notin\mathcal{I}$, the polynomial
approximation problem has closed-form solution
\begin{equation}
\min_{\begin{subarray}{c}
p\in\P_{k}\\
p(1)=1
\end{subarray}}\max_{z\in\mathcal{I}}|p(z)|=\frac{1}{|T_{k}(\frac{1-c}{a})|}\le2\left(\frac{\sqrt{\kappa_{I}}-1}{\sqrt{\kappa_{I}}+1}\right)^{k},
\end{equation}
where $T_{k}(z)$ is the degree-$k$ Chebyshev polynomial of the first
kind, and $\kappa_{I}=(|1-c|+a)/(|1-c|-a)$ is the condition number
for the interval. The minimum is attained by the Chebyshev polynomial
$p^{\star}(z)=T_{k}(\frac{z-c}{a})/|T_{k}(\frac{1-c}{a})|$.\end{thm}
\begin{proof}
See e.g.~\cite{rivlin1974chebyshev}.
\end{proof}
Whereas approximating a general $\kappa$-conditioned region to $\epsilon$-accuracy
requires an order $O(\kappa\log\epsilon^{-1})$ polynomial, approximating
a real interval of the same conditioning and to the same accuracy
only requires an order $O(\sqrt{\kappa}\log\epsilon^{-1})$ polynomial.
This is the underlying mechanism that grants the conjugate gradients
method a square-root factor speed-up over gradient descent; cf.~\cite[Ch. 3]{greenbaum1997iterative}
for a more detailed discussion. For future reference, we also note
the following identity.
\begin{rem}
\label{rem:cheb_bounds}Given any $\zeta>+1$, define the corresponding
condition number as $\nu=(\zeta+1)/(\zeta-1)$. Then
\begin{equation}
|\zeta|^{k}=\left(\frac{\nu+1}{\nu-1}\right)^{k},\qquad\frac{1}{2}\left(\frac{\sqrt{\nu}+1}{\sqrt{\nu}-1}\right)^{k}\le|T_{k}(\zeta)|\le\left(\frac{\sqrt{\nu}+1}{\sqrt{\nu}-1}\right)^{k}.
\end{equation}

\end{rem}

\subsection{Real intervals symmetric about the imaginary axis}

Now, consider the polynomial approximation problem for two real, non-overlapping
intervals with respect to the constraint point $+1$, illustrated
in Fig.~\ref{fig:segseg}, which arises as the eigenvalue distribution
(\ref{eq:eig_case3}) in Lemma~\ref{lem:eig_distr}. 
\begin{lem}
\label{lem:seg_seg}Given $a\ge0$ and $c\ge a$, define the two closed
intervals 
\begin{equation}
\mathcal{I}_{-}=\{z\in\R:|z+c|\le a\},\qquad\mathcal{I}_{+}=\{z\in\R:|z-c|\le a\},
\end{equation}
such that $+1\notin I_{+}$ and $\mathcal{I}_{-}\cap\mathcal{I}_{+}=\emptyset$.
Then the following holds
\begin{equation}
\left(\frac{\sqrt{\kappa_{+}}+1}{\sqrt{\kappa_{+}}-1}\right)^{k}\le\min_{\begin{subarray}{c}
p\in\P_{k}\\
p(1)=1
\end{subarray}}\max_{z\in\mathcal{I}_{-}\cup\mathcal{I}_{+}}|p(z)|\le2\left(\frac{\sqrt{\kappa_{+}}+1}{\sqrt{\kappa_{+}}-1}\right)^{0.317\, k}
\end{equation}
where $\kappa_{+}=(1-c+a)/(1-c-a)$ is the condition number for the
segment $\mathcal{I}_{+}$.
\end{lem}
Of course, the union of the two intervals, i.e. $\mathcal{I}_{-}\cup\mathcal{I}_{+}$,
lies within a single real interval with condition number $\kappa_{I}=(1+c+a)/(1-c-a)$,
so Theorem~\ref{thm:seg_estim} can also be used to obtain an estimate.
However, the explicit treatment of clustering in Lemma~\ref{lem:seg_seg}
yields a considerably tighter bound, because it is entirely possible
for each $\mathcal{I}_{-}$ and $\mathcal{I}_{+}$ to be individually
well-conditioned while admitting an extremely ill-conditioned union.
For a concrete example, consider setting $a=0$ and taking the limit
$c\to1$; the condition number for $\mathcal{I}_{+}$ is fixed at
$\kappa_{+}=1$, but the condition number for the union $\mathcal{I}_{-}\cup\mathcal{I}_{+}$
diverges $\kappa_{I}\to\infty$. In this case, Lemma~\ref{lem:seg_seg}
predicts extremely rapid convergence for all values of $c$, whereas
Theorem~\ref{thm:seg_estim} does not promise convergence at all.
\begin{figure}
\hfill{}\includegraphics[width=0.5\columnwidth]{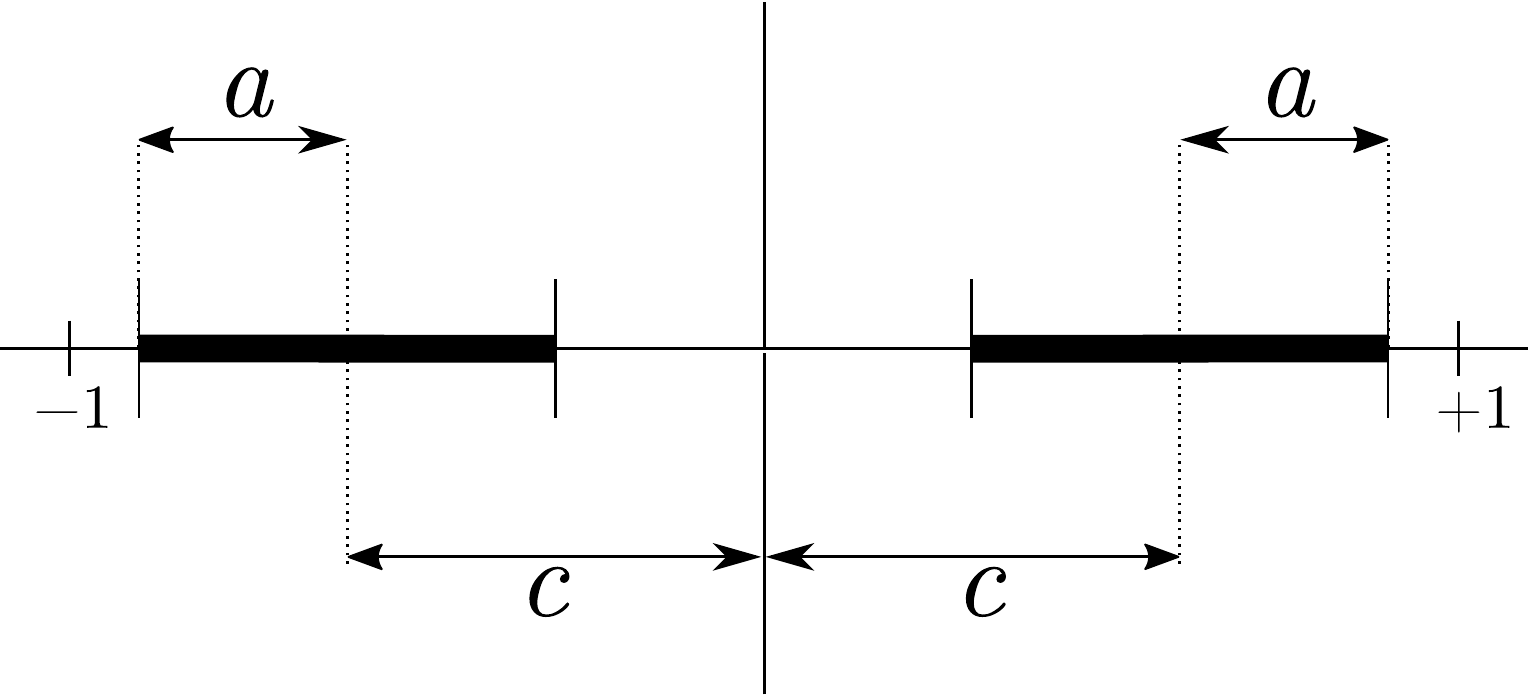}\hfill{}

\protect\caption{\label{fig:segseg}Real intervals symmetric about the imaginary axis.}
\end{figure}

To prove Lemma~\ref{lem:seg_seg}, we will begin by stating a technical
lemma.
\begin{lem}
\label{lem:ratio_monotonous}Define $f(x)=\log[(x-1)/(x+1)]$ with
domain $x\in(1,\infty)$. Then the quotient $g(x)=f(x)/f(x^{2})$
is monotonously increasing with infimum attained at the limit point
$g(1)=1$.\end{lem}
\begin{proof}
By definition, we see that both $f(x)$ and $f(x^{2})$ are nonzero
for all $x>1$. Taking the derivatives 
\begin{equation}
\frac{d}{dx}\left[f(x)\right]=\frac{2}{x^{2}-1}=\frac{2x^{2}+2}{x^{4}-1},\qquad\frac{d}{dx}\left[f(x^{2})\right]=\frac{4x}{x^{4}-1},
\end{equation}
reveals that $f(x)$ is monotonously increasing for all $x>1$, so
we also have $f(x)<f(x^{2})$. Finally, we observe that $\frac{d}{dx}\left[f(x)\right]>\frac{d}{dx}\left[f(x^{2})\right]>0$
for all $x>1$. Combining these three observations with the quotient
rule reveals that $g(x)$ is monotonously increasing
\begin{equation}
\frac{d}{dx}\left[g(x)\right]=\frac{f(x^{2})\,\frac{d}{dx}\left[f(x)\right]-f(x)\,\frac{d}{dx}\left[f(x^{2})\right]}{[f(x^{2})]^{2}}>0\qquad\forall x>1.
\end{equation}
Hence, the infimum for $g(x)$ must be attain at its lower limit point
$x=1$. Using l'H\^{o}pital's rule yields $\lim_{x\to1}g(x)=\lim_{x\to1}(2x^{2}+2)/(4x)=1$.
\end{proof}

\begin{proof}[Proof of Lemma~\ref{lem:seg_seg}]
For the lower-bound, we have via Theorem~\ref{thm:seg_estim} and
Remark~\ref{rem:cheb_bounds}
\[
\min_{\begin{subarray}{c}
p\in\P_{k}\\
p(1)=1
\end{subarray}}\max_{z\in\mathcal{I}_{-}\cup\mathcal{I}_{+}}|p(z)|\ge\min_{\begin{subarray}{c}
p\in\P_{k}\\
p(1)=1
\end{subarray}}\max_{z\in\mathcal{I}_{+}}|p(z)|=\frac{1}{T_{k}(\frac{1-c}{a})}\ge\left(\frac{\sqrt{\kappa_{+}}+1}{\sqrt{\kappa_{+}}-1}\right)^{k}.
\]
For the upper-bound, consider the product of an order-$\xi$ Chebyshev
polynomial over $\mathcal{I}_{+}$ and an order-$\eta$ monomial over
$\mathcal{I}_{-}$, as in
\begin{equation}
p(z)=\left(\frac{z+c}{1+c}\right)^{\eta}\frac{T_{\xi}(\frac{z-c}{a})}{|T_{\xi}(\frac{1-c}{a})|},
\end{equation}
with infinity norms $\|p(z)\|_{\mathcal{I}_{-}}\triangleq\max_{z\in\mathcal{I}_{-}}|p(z)|$
and $\|p(z)\|_{\mathcal{I}_{+}}\triangleq\max_{z\in\mathcal{I}_{+}}|p(z)|$
attained at $z=-(c+a)$ and $z=+(c+a)$ respectively
\begin{equation}
\|p(z)\|_{\mathcal{I}_{-}}=\left(\frac{a}{1+c}\right)^{\eta}\frac{|T_{\xi}(\frac{a+2c}{a})|}{|T_{\xi}(\frac{1-c}{a})|},\qquad\|p(z)\|_{\mathcal{I}_{+}}=\left(\frac{a+2c}{1+c}\right)^{\eta}\frac{1}{|T_{\xi}(\frac{1-c}{a})|}.
\end{equation}
We choose the exponents $\eta+\xi=k$ in the ratio
\begin{align}
\eta/\xi & =\log\left(\frac{\sqrt{\nu}-1}{\sqrt{\nu}+1}\right)\bigg/\log\left(\frac{\nu-1}{\nu+1}\right),\label{eq:eta_xi}
\end{align}
in which $\nu=1+a/c$ is the condition number of the well-conditioned
interval $\mathcal{I}_{-}$ with respect to the ill-conditioned interval
$\mathcal{I}_{+}$. This particular ratio implies 
\begin{equation}
\left(\frac{\nu-1}{\nu+1}\right)^{\eta}=\left(\frac{\sqrt{\nu}-1}{\sqrt{\nu}+1}\right)^{\xi}\qquad\implies\qquad\left(\frac{a}{a+2c}\right)^{\eta}\le\frac{1}{|T_{\xi}(\frac{a+2c}{a})|},\label{eq:q_pow}
\end{equation}
via the bounds in Remark~\ref{rem:cheb_bounds}, so $\|p(z)\|_{\mathcal{I}_{+}}\ge\|p(z)\|_{\mathcal{I}_{-}}$
is satisfied by construction, and the global error bound is bound
\begin{equation}
\max_{z\in\mathcal{I}_{-}\cup\mathcal{I}_{+}}|p(z)|\le\|p(z)\|_{\mathcal{I}_{+}}\le2\left(\frac{\sqrt{\kappa_{+}}+1}{\sqrt{\kappa_{+}}-1}\right)^{\xi}.
\end{equation}
To complete the proof, we require a lower estimate of $\xi$ in terms
of $k$ that is valid for any valid of $a$ and $c$, or equivalently,
any value of $\nu$. Since $k=\eta+\xi$ by definition, the ratio
is written $\xi/k=1/(1+\eta/\xi)$, so we really desire an upper estimate
on the ratio $\eta/\xi$ defined in (\ref{eq:eta_xi}). According
to Lemma~\ref{lem:ratio_monotonous}, the quotient in (\ref{eq:eta_xi})
is monotonously increasing with respect to $\sqrt{\nu}$. Hence, the
maximum value of $\eta/\xi$ is attained at the maximum value of $\nu$.
The choice of $a=c$ maximizes $\nu$ with maximum at $v=2$, since
any choice of $a>c$ would cause $\mathcal{I}_{-}$ and $\mathcal{I}_{+}$
to overlap. Evaluating the expression at $\nu=2$ yields $\eta/\xi\le\log\left(\frac{\sqrt{2}-1}{\sqrt{2}+1}\right)\bigg/\log\left(\frac{2-1}{2+1}\right)\le2.151$.
This implies $\xi/k=1/(1+\eta/\xi)\ge1/3.151=0.317$.
\end{proof}

\subsection{Concentric disk and interval}

\begin{figure}
\hfill{}\includegraphics[width=0.5\columnwidth]{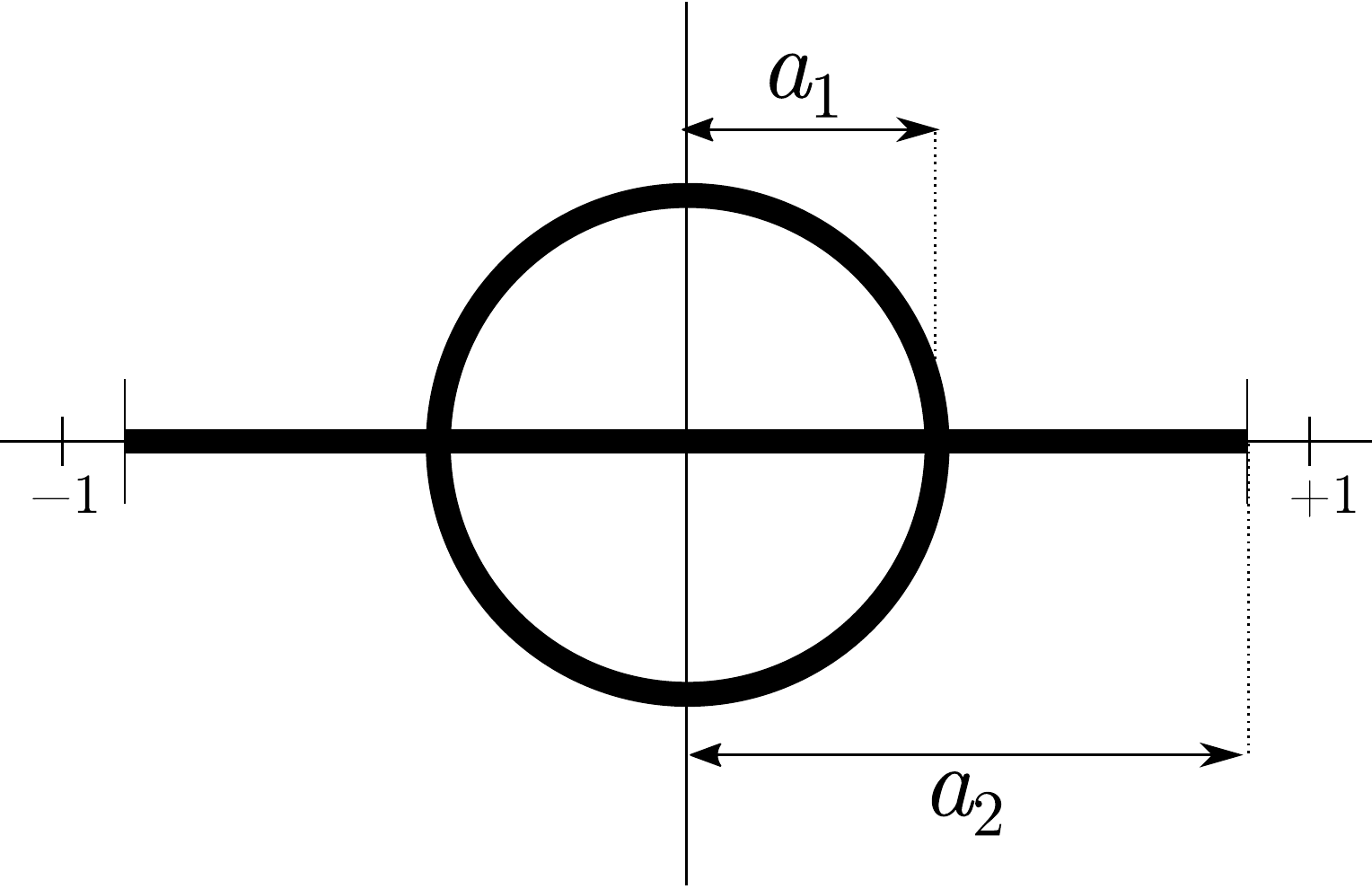}\hfill{}

\protect\caption{\label{fig:sidephi}Disk-and-interval eigenvalue distribution.}
\end{figure}
Finally, consider the polynomial approximation problem for the union
of a disk and a real interval with respect to the constraint point
$+1$, illustrated in Fig.~\ref{fig:sidephi}, which arises as the
eigenvalue distribution (\ref{eq:eig_case1}) in Lemma~\ref{lem:eig_distr}. 
\begin{lem}
\label{lem:disk_inter}Given $0\le a_{D}\le a_{I}<1$, define the
disk $\mathcal{D}=\{z\in\C:|z|\le a_{D}\}$ and the interval $\mathcal{I}=\{z\in\C:|z|\le a_{I}\}$.
Then 
\begin{equation}
\min_{\begin{subarray}{c}
p\in\P_{k}\\
p(1)=1
\end{subarray}}\max_{z\in\mathcal{D}\cup\mathcal{I}}|p(z)|\le2\left(\frac{\kappa_{I}-1}{\kappa_{I}+1}\right)^{\eta}\left(\frac{\sqrt{\kappa_{I}}-1}{\sqrt{\kappa_{I}}+1}\right)^{\xi}.\label{eq:disk_seg_optim_constr}
\end{equation}
where $\kappa_{I}=(1+a_{I})/(1-a_{I})$ is the condition number for
the interval, and $\eta+\xi=k$ are defined
\[
\eta=\left\lceil \frac{c_{0}}{\delta+c_{0}}\right\rceil ,\qquad\xi=\left\lfloor \frac{\delta}{\delta+c_{0}}\right\rfloor 
\]
where $\delta=1-a_{D}/a_{I}$ and $c_{0}=\log(1+\sqrt{2})\approx0.8814$.
\end{lem}
It is easy to verify that $\kappa_{I}$ is also the condition number
for the union of the disk and the interval, $\mathcal{D}\cup\mathcal{I}$.
Consequently, one interpretation of Lemma~\ref{lem:disk_inter} is
that there exists an order $O(\sqrt{\kappa}\log\epsilon^{-1})$ polynomial
that approximates a $\kappa$-conditioned version of $\mathcal{D}\cup\mathcal{I}$
to $\epsilon$-accuracy, but only so long as the disk $\mathcal{D}$
is \emph{strictly better conditioned} than the interval $\mathcal{I}$.
If $a_{D}=a_{I}$, then both regions share the same condition number,
and the square-root factor speed-up is lost; an order $O(\kappa\log\epsilon^{-1})$
polynomial is now required to solve the same approximation problem.

The proof of Lemma~\ref{lem:disk_inter} requires the following estimate
on the value of the Chebyshev polynomial over the complex plane.
\begin{prop}
\label{prop:chebmax}The maximum modulus of the $k$-th order Chebyshev
polynomial is bound within the disk on the complex plane centered
at the origin with radius $\eta$, 
\begin{equation}
\max_{|z|\le\eta}|T_{k}(z)|\le T_{k}(\sqrt{1+\eta^{2}})\le\left(\eta+\sqrt{1+\eta^{2}}\right)^{k},\label{eq:chebmaxbnd}
\end{equation}
and the first inequality is tight for $k$ even.\end{prop}
\begin{proof}
The maximum modulus for $T_{k}(z)$ over the ellipse with unit focal
distance and principal axis $a\ge\eta$ are attained at $2n$ points
along its boundary the points~\cite{rivlin1974chebyshev,saad1986gmres}
\begin{equation}
z_{k}=a\cos\left(\frac{k\pi}{n}\right)+j\sqrt{a^{2}-1}\sin\left(\frac{k\pi}{n}\right)\qquad k=1,\ldots,2n.
\end{equation}
The ellipse with $a=\sqrt{1+\eta^{2}}$ is the smallest to enclose
the disk of radius $\eta$, and if $k$ is even, then $z_{k/2}$ also
lies on its boundary. The second bound follows by definition
\begin{equation}
T_{k}(\sqrt{1+\eta^{2}})=\frac{1}{2}\left(\eta+\sqrt{1+\eta^{2}}\right)^{k}+\frac{1}{2}\left(\eta+\sqrt{1+\eta^{2}}\right)^{-k}\le\left(\eta+\sqrt{1+\eta^{2}}\right)^{k}.
\end{equation}

\end{proof}

\begin{proof}[Proof of Lemma~\ref{lem:disk_inter}]
Consider the product of an order-$\xi$ Chebyshev polynomial over
$\mathcal{I}_{+}$ and an order-$\eta$ monomial over $\mathcal{I}_{-}$,
as in
\begin{equation}
p(z)=\left(\frac{z/a_{D}}{1/a_{D}}\right)^{\eta}\frac{T_{\xi}(z/a_{I})}{|T_{\xi}(1/a_{I})|}=z^{\eta}\,\frac{T_{\xi}(z/a_{I})}{|T_{\xi}(1/a_{I})|},
\end{equation}
with infinity norms $\|p(z)\|_{\mathcal{D}}\triangleq\max_{z\in\mathcal{D}}|p(z)|$
and $\|p(z)\|_{\mathcal{I}}\triangleq\max_{z\in\mathcal{I}}|p(z)|$
given
\begin{equation}
\|p(z)\|_{\mathcal{D}}=\frac{a_{D}^{\eta}\,\|T_{\xi}(z/a_{I})\|_{\mathcal{D}}}{|T_{\xi}(1/a_{I})|}\le\frac{a_{D}^{\eta}(1+\sqrt{2})^{\xi}}{|T_{\xi}(1/a_{I})|},\qquad\|p(z)\|_{\mathcal{I}}=\frac{a_{I}^{\eta}}{|T_{\xi}(1/a_{I})|}.\label{eq:max_D}
\end{equation}
The bound $(1+\sqrt{2})^{\xi}\ge\max_{|z|\le1}|T_{\xi}(z)|\ge\|T_{\xi}(z/a_{I})\|_{\mathcal{D}}$
arises from $a_{D}\le a_{I}$ and Proposition~\ref{prop:chebmax}.
Choosing the exponents $\eta$ and $\xi$ to satisfy the ratio 
\begin{equation}
\eta/\xi=\frac{\log(1+\sqrt{2})}{1-a_{D}/a_{I}}\ge\frac{\log(1+\sqrt{2})}{\log(a_{I}/a_{D})}\qquad\implies\qquad\left(\frac{a_{D}}{a_{I}}\right)^{\eta}\le\frac{1}{(1+\sqrt{2})^{\xi}},
\end{equation}
satisfies $\|p(z)\|_{\mathcal{D}}\le\|p(z)\|_{\mathcal{I}}$ by construction,
so the global error is bound by $\|p(z)\|_{\mathcal{I}}$. Bounding
the term $1/|T_{\xi}(1/a_{I})|$ in $\|p(z)\|_{\mathcal{I}}$ with
Remark~\ref{rem:cheb_bounds} completes the result.
\end{proof}

\section{\label{sec:Proof-main}Proof of the Main Results}

With the eigenvalues of $K(\beta)$ characterized and the corresponding
approximation problems solved, we are now ready to prove our main
results.

\subsection{Regime of purely-real eigenvalues}

Loosely speaking, Theorem~\ref{thm:real_eig} states that for parameter
values of $\beta>\ell$ or $\beta<m$, ADMM-GMRES is guaranteed to
converge to an $\epsilon$-accurate solution in $O(\sqrt{\kappa}\log\epsilon^{-1})$
iterations. We will prove this statement by solving the $K(\beta)$
eigenvalue approximation problem associated for $\beta>\ell$ or $\beta<m$
heuristically, using Theorem~\ref{thm:seg_estim} and Lemma~\ref{lem:seg_seg},
and substituting the resulting bound into Proposition~\ref{prop:admm_gmres}.
This two-step process begins with the following bound.
\begin{lem}
\label{lem:approx_real}Let $\beta>\ell$ or $\beta<m$. Then the
eigenvalue approximation problem for $K(\beta)$ has bounds
\begin{equation}
\min_{\begin{subarray}{c}
p\in\P_{k}\\
p(1)=1
\end{subarray}}\max_{\lambda\in\Lambda\{K(\beta)\}}|p(\lambda)|\le2\left(\frac{\sqrt{2\kappa}-1}{\sqrt{2\kappa}+1}\right)^{0.317\, k}.\label{eq:preal_thm}
\end{equation}
\end{lem}
\begin{proof}
First, we consider $\beta\in[\frac{1}{2}m,m)\cup(\ell,2\ell]$. According
to (\ref{eq:eig_case2}) in Lemma~\ref{lem:eig_distr}, the eigenvalues
of $K(\beta)$ are distributed over a purely-real interval bounded
by $\|K(\beta)\|=(\gamma-1)/(\gamma+1)$, where $\gamma=\max\{\beta/m,\ell/\beta\}$
lies $\gamma\in(\kappa,2\kappa]$. The associated condition number
$\kappa_{I}=\gamma$ is bounded $\kappa<\kappa_{I}\le2\kappa$, and
applying the Chebyshev polynomial approximation in Theorem~\ref{thm:seg_estim}
yields a less conservative version of (\ref{eq:preal_thm}), i.e.
one with a larger exponent on the upper-bound.

Next, we consider the remaining choices, $\beta>2\ell$ or $\beta<\frac{1}{2}m$.
According to (\ref{eq:eig_case3}) in Lemma~\ref{lem:eig_distr},
the eigenvalues of $K(\beta)$ are clustered along two non-overlapping
intervals, symmetric about the imaginary axis. The condition number
for the interval lying in the right-half plane is $\kappa_{+}=\frac{3}{2}\kappa(\gamma+1)/(\gamma+\kappa)$
with $\gamma>2\kappa$, so its value is bound $\kappa<\kappa_{+}\le2\kappa$.
Making this substitution into the heuristic solution for the two-segment
problem in Lemma~\ref{lem:seg_seg} yields exactly (\ref{eq:preal_thm}).
\end{proof}
According to Lemma~\ref{eq:preal_thm}, solving the $K(\beta)$ eigenvalue
approximation problem (for $\beta>\ell$ or $\beta<m$) to $\epsilon$-accuracy
will require a polynomial of order $\approx2.2\sqrt{\kappa}\log\epsilon^{-1}$,
where the leading constant is $2.2\approx\sqrt{2}/(2\times0.317)$.
Assuming that the eigenvalue characterizations (\ref{eq:eig_case2})
and (\ref{eq:eig_case3}) in Lemma~\ref{lem:eig_distr} are sharp,
this figure cannot be improved by more than a small absolute constant.
This is because all other ingredients in the proof of Lemma~\ref{lem:approx_real}
have approximation constants no greater than 4.

\begin{proof}[Proof of Theorem~\ref{thm:real_eig}]
Substituting the bound on the eigenvalue approximation problem in
Lemma~\ref{lem:approx_real} and the bound on the condition number
for the matrix-of-eigenvectors in Lemma~\ref{lem:eig_real} into
Proposition~\ref{prop:admm_gmres} yields the desired statement.
\end{proof}

\subsection{Regime of complex eigenvalues}

Loosely speaking, Theorem~\ref{thm:real_eig} states that for parameter
values of $m\le\beta\le\ell$, ADMM-GMRES is guaranteed to converge
to an $\epsilon$-accurate solution in $O(\kappa^{2/3}\log\epsilon^{-1})$
iterations. We will prove this statement by solving the $K(\beta)$
eigenvalue approximation problem heuristically using Lemma~\ref{lem:disk_inter},
estimating a bound on the resulting convergence factor that is independent
of $\beta$, and substituting the bound into Proposition~\ref{prop:admm_gmres}.

To begin, Lemma~\ref{lem:eig_distr} says that for $\beta\in[m,\ell]$,
the eigenvalues of $K(\beta)$ are distributed over the union of a
disk and an interval. Lemma~\ref{lem:disk_inter} can be used to
provide a heuristic solution for this eigenvalue distribution. Substituting
the former into the latter yields
\begin{equation}
\min_{\begin{subarray}{c}
p\in\P_{k}\\
p(1)=1
\end{subarray}}\max_{\lambda\in K(\beta)}|p(\lambda)|\le2\rho^{k},\label{eq:expand_exp}
\end{equation}
where the convergence factor and associated parameters are
\begin{equation}
\rho(\gamma)=\left(\frac{\gamma-1}{\gamma+1}\right)^{\frac{c_{0}}{c_{0}+\delta}}\left(\frac{\sqrt{\gamma}-1}{\sqrt{\gamma}+1}\right)^{\frac{\delta}{c_{0}+\delta}},\qquad\delta(\gamma)=\frac{\gamma^{2}-\kappa}{(\gamma+\kappa)(\gamma+1)}\label{eq:thm_conv_fac}
\end{equation}
and $c_{0}=\log(1+\sqrt{2})$. Recall that $\gamma=\max\{\beta/m,\ell/\beta\}$. 
\begin{lem}
\label{lem:approx_cmp}Let $\sqrt{\kappa}\le\gamma\le\kappa$, and
define $\delta(\gamma)$ as a function of $\gamma$ as in (\ref{eq:thm_conv_fac}).
Then 
\begin{equation}
\left(\frac{\gamma-1}{\gamma+1}\right)^{\frac{c_{0}}{c_{0}+\delta(\gamma)}}\left(\frac{\sqrt{\gamma}-1}{\sqrt{\gamma}+1}\right)^{\frac{\delta(\gamma)}{c_{0}+\delta(\gamma)}}\le\left(1-\frac{1}{\kappa^{2/3}+1}\right)^{0.209\, k}.\label{eq:cmp_thm}
\end{equation}
\end{lem}
\begin{proof}
We will attempt to lower-bound the logarithm of (\ref{eq:cmp_thm})
by applying $\log(1+x)\le x$, as in
\begin{equation}
-\frac{1}{2}\log\rho\ge\frac{c_{0}}{c_{0}+\delta(\gamma)}\left(\frac{1}{\gamma+1}\right)+\frac{\delta(\gamma)}{c_{0}+\delta(\gamma)}\left(\frac{1}{\sqrt{\gamma}+1}\right).\label{eq:proof_thm_log_conv_fac}
\end{equation}
Directly substituting $\delta(\gamma)$ from (\ref{eq:thm_conv_fac})
into (\ref{eq:proof_thm_log_conv_fac}) and sweeping $\gamma=\kappa^{\alpha}$
over $\alpha\in[0.5,1]$ yields
\begin{align}
\text{(\ref{eq:proof_thm_log_conv_fac})}= & \frac{c_{0}^{-1}\kappa^{2\alpha}+\kappa^{1.5\alpha}+\kappa^{\alpha}+\kappa^{0.5\alpha+1}-(c_{0}^{-1}-1)\kappa}{\left(\kappa^{0.5\alpha}+1\right)\left((c_{0}^{-1}+1)\kappa^{2\alpha}+\kappa^{\alpha}+\kappa^{\alpha+1}-(c_{0}^{-1}-1)\kappa\right)}\label{eq:prf_thm_step1}\\
= & \frac{c_{0}^{-1}\kappa^{0.5\alpha-1}+\kappa^{-1}+\kappa^{-0.5\alpha-1}+\kappa^{-\alpha}-(c_{0}^{-1}-1)\kappa^{-1.5\alpha}}{\left(1+\kappa^{-0.5\alpha}\right)\left((c_{0}^{-1}+1)\kappa^{\alpha-1}+\kappa^{-1}+1-(c_{0}^{-1}-1)\kappa^{-\alpha}\right)}\label{eq:prf_thm_step2}\\
\ge & \frac{c_{0}^{-1}\kappa^{0.5\alpha-1}+0+0+\kappa^{-\alpha}-(c_{0}^{-1}-1)\kappa^{-1.5\alpha}+0}{\left(1+1\right)\left((c_{0}^{-1}+1)+1+1\right)}\label{eq:proof_thm_ineq}\\
= & \frac{c_{0}^{-1}\kappa^{0.5\alpha-1}+\kappa^{-\alpha}-(c_{0}^{-1}-1)\kappa^{-1.5\alpha}}{2(c_{0}^{-1}+3)}\label{eq:prf_thm_step4}
\end{align}
where (\ref{eq:prf_thm_step1})$\Rightarrow$(\ref{eq:prf_thm_step2})
divides the numerator and denominator by $\kappa^{1.5\alpha+1}$,
and (\ref{eq:prf_thm_step2})$\Rightarrow$(\ref{eq:proof_thm_ineq})
uses the fact that $\kappa\ge1$ and that $(c_{0}^{-1}-1)>0$. Finally,
employing the bounds
\begin{align*}
c_{0}^{-1}\kappa^{0.5\alpha-1}+\kappa^{-\alpha} & \ge\min\{c_{0}^{-1},1\}\max\{\kappa^{0.5\alpha-1},\kappa^{-\alpha}\}\ge\kappa^{-2/3}\quad\forall\alpha>0,\\
(c_{0}^{-1}-1)\kappa^{-1.5\alpha} & \le(c_{0}^{-1}-1)\kappa^{-0.75}\le(c_{0}^{-1}-1)\kappa^{-2/3}\quad\forall\alpha\in[0.5,1],
\end{align*}
simplifies (\ref{eq:prf_thm_step4}) enough for an exact number
\[
-\log\rho\ge2\,\frac{2-c_{0}^{-1}}{2(c_{0}^{-1}+3)}\kappa^{-2/3}\ge0.209\kappa^{-2/3}.
\]
Finally, we translate the lower-bound $-\log\rho\ge c\kappa^{-\alpha}$
into the upper-bound $\rho\le(1-(\kappa^{\alpha}+1)^{-1})^{c}$ using
the fact that $\log(1+x)\ge x(1+x)^{-1}$, thereby completing the
proof.
\end{proof}

\begin{proof}[Proof of Theorem~\ref{thm:cmp_eig}]
Substituting the convergence factor bound in Lemma~\ref{lem:approx_cmp}
into (\ref{eq:expand_exp})-(\ref{eq:thm_conv_fac}) yields a bound
to the eigenvalue approximation problem
\begin{equation}
\min_{\begin{subarray}{c}
p\in\P_{k}\\
p(1)=1
\end{subarray}}\max_{\lambda\in K(\beta)}|p(\lambda)|\le2\left(1-\frac{1}{\kappa^{2/3}+1}\right)^{0.209\, k}\label{eq:cmp_bnd}
\end{equation}
for any $m\le\beta\le\ell$. Substituting the bound (\ref{eq:cmp_bnd})
into Proposition~\ref{prop:admm_gmres} proves the desired statement.
\end{proof}

\section{\label{sec:AppEx}Numerical Examples}

Finally, we benchmark the performance of ADMM-GMRES numerically. Two
classes of problems are considered: (1) random problems generated
by selecting random orthonormal bases and singular values; and (2)
the Newton direction subproblems associated with the interior-point
solution of large-scale semidefinite programs. 

In each case, the parameter value $\beta$ used in ADMM-GMRES is randomly
selected from the log-uniform distribution scaled to span four orders
of magnitude, from $10^{-2}$ to $10^{2}$. More precisely, let $Y$
be a random variable uniformly distributed in $[-1,+1]$; then for
each subproblem solved, we randomly select $\beta$ from the distribution
$10^{2Y}$. These results are benchmarked against regular ADMM with
the optimal parameter value of $\beta=\sqrt{m\ell}$.

Overall, the numerical results validate our conclusions. We find that
ADMM-GMRES converges to an $\epsilon$-accurate solution of a $\kappa$-conditioned
problem within $O(\sqrt{\kappa}\log\epsilon^{-1})$ iterations. This
is a slightly stronger finding than our theoretical predictions, which
only promised convergence in $O(\kappa^{2/3}\log\epsilon^{-1})$ iterations.

\subsection{Random problems}

\begin{figure}
\hfill{}\includegraphics[width=0.6\columnwidth]{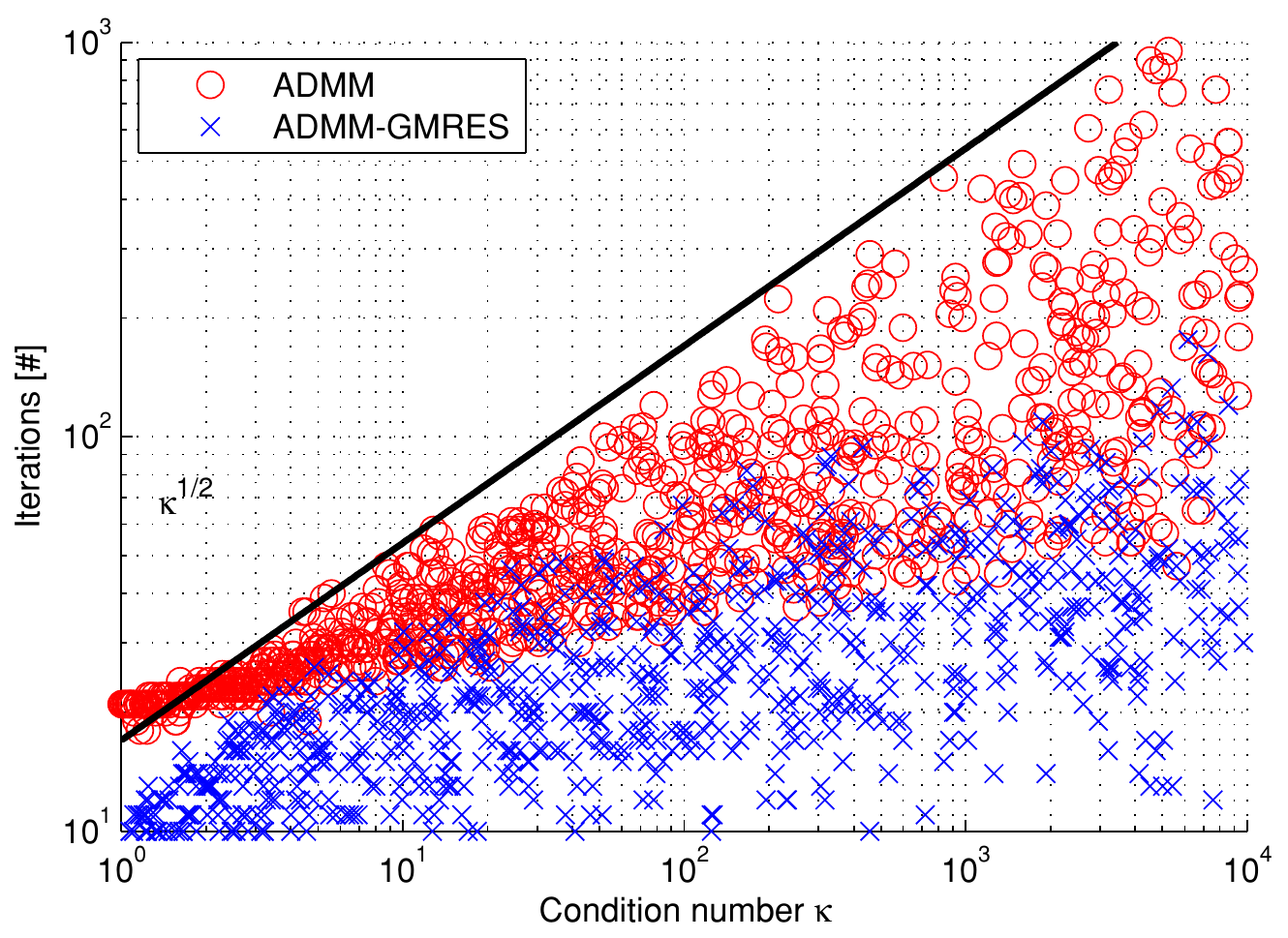}\hfill{}

\protect\caption{\label{fig:rand_prob}Number of iterations to solve 1000 randomly
generated problems to $\epsilon=10^{-6}$ accuracy using ADMM with
$\beta=\sqrt{m\ell}$ (circles) and GMRES-accelerated ADMM with randomly
selected $\beta$ (crosses). The solid line is $17\sqrt{\kappa}$.
Both methods converge in $O(\sqrt{\kappa})$ iterations. The problems
have random dimensions $1\le\protect\nx\le1000$, $1\le\protect\ny\le\protect\nx$,
$1\le\protect\nz\le\protect\ny$. }
\end{figure}
First, we compare the performance of ADMM and GMRES-accelerated ADMM
in the solution of random problems generated via the following procedure
taken from~\cite{zhang2016on}.

\begin{constr}\label{constr:log-normal}Begin with nonzero positive
integer parameters $\nx$, $\ny\le\nx$, $\nz\le\ny$ and positive
real parameter $s$. 
\begin{enumerate}
\item Select the orthogonal matrices $U_{A},U_{B}\in\R^{\ny\times\ny}$,
$V_{A},U_{D}\in\R^{\nx\times\nx}$, $V_{B}\in\R^{\ny\times\nz}$ i.i.d.
uniformly from their respective orthogonal groups. 
\item Select the positive scalars $\sigma_{A}^{(1)},\ldots,\sigma_{A}^{(\ny)}$,
$\sigma_{B}^{(1)},\ldots,\sigma_{B}^{(\nz)}$, and $\sigma_{D}^{(1)},\ldots,\sigma_{D}^{(\nx)}$
i.i.d. from the log-normal distribution $\sim\mathrm{exp}(0,s^{2})$. 
\item Output the matrices $A=U_{A}\mathrm{diag}(\sigma_{A}^{(1)},\ldots,\sigma_{A}^{(\ny)})V_{A}^{T}$,
$B=U_{B}\mathrm{diag}(\sigma_{B}^{(1)},\ldots,\sigma_{B}^{(\ny)})V_{B}^{T}$,
and $D=U_{D}\mathrm{diag}(\sigma_{D}^{(1)},\ldots,\sigma_{D}^{(\ny)})U_{D}^{T}$.
\end{enumerate}
\end{constr}

The dimension parameters $\nx$, $\ny$, $\nz$ are uniformly sampled
from $\nx\in\{1,\ldots,1000\}$, $\ny\in\{1,\ldots,\nx\}$, and $\nz\in\{1,\ldots,\nz\}$,
and the log-standard-deviation uniformly swept within the range $s\in[0,1]$,
in order to produce a range of condition numbers spanning $1\le\kappa\le10^{4}$.
Note that by construction, the optimal parameter choice $\sqrt{m\ell}$
has an expected value of 1.

Figure~\ref{fig:rand_prob} plots the number of iterations to converge
to $\epsilon=10^{-6}$ for each method and over each problem. We see
that both ADMM and ADMM-GMRES converges in $O(\sqrt{\kappa})$ iterations,
with ADMM-GMRES typically converging in slightly fewer iterations
than ADMM. The difference, of course, is that the feat is achieved
by ADMM-GMRES without needing to estimate the values of $m$ and $\ell$.

Note that the ADMM-GMRES curve bends downwards with increasing $\kappa$.
This is an artifact of the distribution of $\beta$ becoming optimal
with increasing $\kappa$. As we noted in the proof of our main results,
the convergence of ADMM-GMRES is entirely driven by an indirect, rescaled
quantity $\gamma=\max\{\ell/\beta,\beta/m\}$. When $\ell$ and $m$
are increased and decreased at the same uniform rate, the distribution
for $\gamma$ and $\beta$ become concentrated about $\gamma=\sqrt{\kappa}$
and $\beta=\sqrt{m\ell}$ respectively. These choices of $\gamma$
and $\beta$ often allow ADMM-GMRES to converge in $O(\kappa^{\frac{1}{4}}\log\epsilon^{-1})$
iterations~\cite{zhang2016on}.

\subsection{Interior-point Newton Direction for SDPs}

\begin{figure}
\hfill{}\includegraphics[width=0.6\columnwidth]{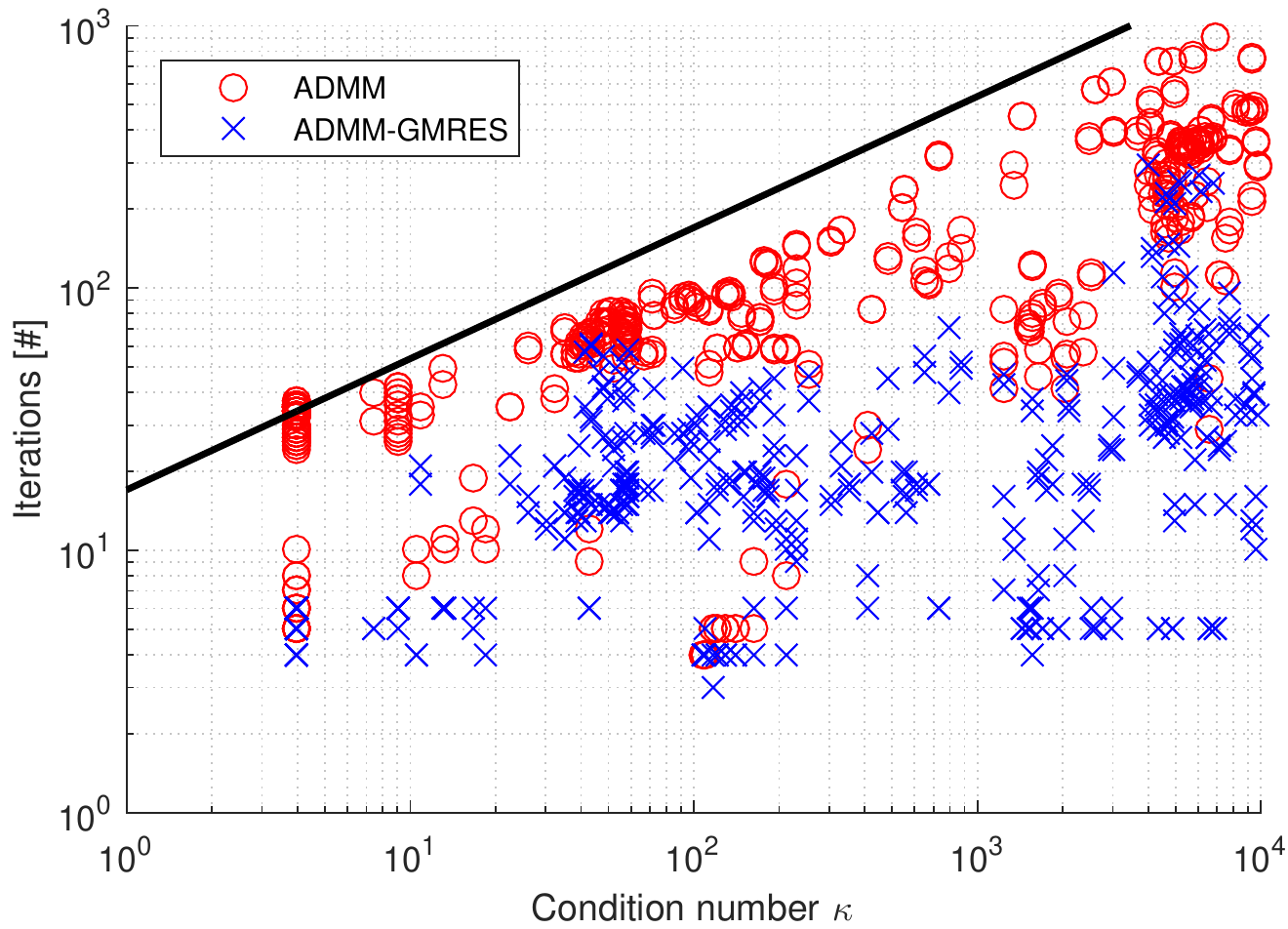}\hfill{}

\protect\caption{\label{fig:aggregate}Number of iterations to solve 508 Newton direction
subproblems with $\kappa\le10^{4}$ to $\epsilon=10^{-6}$ accuracy.
The solid line plots $k=17\sqrt{\kappa}$. }
\end{figure}
Next, we compare the performance of ADMM and GMRES-accelerated ADMM
in their ability to recompute the Newton steps as generated by SeDuMi~\cite{sturm1999using}
over 80 semidefinite programs (SDPs) in the SDPLIB suite~\cite{borchers1999sdplib}.
The experimental set-up is similar to that described in~\cite{zhang2016on}:
the 80 problems from the SDPLIB suite with less than 700 constraints
are pre-solved using SeDuMi, and the predictor and corrector Newton
step problems at each interior-point step are exported, each of the
form 
\begin{equation}
\begin{bmatrix}D & 0 & I\\
0 & 0 & B^{T}\\
I & B & 0
\end{bmatrix}\begin{bmatrix}\Delta x\\
\Delta z\\
\Delta y
\end{bmatrix}=\begin{bmatrix}r_{x}\\
r_{z}\\
r_{y}
\end{bmatrix}.
\end{equation}
Clearly, this is the KKT system for the prototype ADMM problem (\ref{eq:consen2})
with substitutions $f(x)=\frac{1}{2}x^{T}Dx-r_{x}^{T}x$, $g(z)=-r_{z}^{T}z$,
$A=I$, $c=r_{y}$, so both ADMM and ADMM-GMRES can be used to recompute
the solution $[\Delta x;\Delta z;\Delta y]$. The associated matrix-vector
products should be implicitly performed in order for either methods
to be efficient (cf.~\cite[Sec. 8]{zhang2016on}), but this is an
implementation detail that does not affect the iterates generated. 

Figure~\ref{fig:aggregate} shows the number of iterations to converge
to $\epsilon=10^{-6}$ over all 508 Newton direction subproblems with
$\kappa\le10^{4}$. Again, both ADMM and ADMM-GMRES required $O(\sqrt{\kappa})$
iterations, with the latter achieving the feat without needing to
estimate the values of $m$ and $\ell$. In fact, ADMM-GMRES converged
in fewer iterations for all of the problems considered.

\appendix
\bibliographystyle{siam}
\bibliography{refs}

\begin{thebibliography}{10}

\bibitem{bai2003hermitian}
{\sc Z.-Z. Bai, G.~H. Golub, and M.~K. Ng}, {\em Hermitian and skew-hermitian
  splitting methods for non-hermitian positive definite linear systems}, SIAM
  Journal on Matrix Analysis and Applications, 24 (2003), pp.~603--626.

\bibitem{battermann1998preconditioners}
{\sc A.~Battermann and M.~Heinkenschloss}, {\em Preconditioners for
  karush-kuhn-tucker matrices arising in the optimal control of distributed
  systems}, in Control and Estimation of Distributed Parameter Systems,
  Springer, 1998, pp.~15--32.

\bibitem{benzi2004preconditioner}
{\sc M.~Benzi and G.~H. Golub}, {\em A preconditioner for generalized saddle
  point problems}, SIAM Journal on Matrix Analysis and Applications, 26 (2004),
  pp.~20--41.

\bibitem{benzi2005numerical}
{\sc M.~Benzi, G.~H. Golub, and J.~Liesen}, {\em Numerical solution of saddle
  point problems}, Acta numerica, 14 (2005), pp.~1--137.

\bibitem{benzi2006eigenvalues}
{\sc M.~Benzi and V.~Simoncini}, {\em On the eigenvalues of a class of saddle
  point matrices}, Numerische Mathematik, 103 (2006), pp.~173--196.

\bibitem{borchers1999sdplib}
{\sc B.~Borchers}, {\em Sdplib 1.2, a library of semidefinite programming test
  problems}, Optimization Methods and Software, 11 (1999), pp.~683--690.

\bibitem{boyd2011distributed}
{\sc S.~Boyd, N.~Parikh, E.~Chu, B.~Peleato, and J.~Eckstein}, {\em Distributed
  optimization and statistical learning via the alternating direction method of
  multipliers}, Foundations and Trends{\textregistered} in Machine Learning, 3
  (2011), pp.~1--122.

\bibitem{bramble1997analysis}
{\sc J.~H. Bramble, J.~E. Pasciak, and A.~T. Vassilev}, {\em Analysis of the
  inexact uzawa algorithm for saddle point problems}, SIAM Journal on Numerical
  Analysis, 34 (1997), pp.~1072--1092.

\bibitem{brown1994convergence}
{\sc P.~N. Brown and Y.~Saad}, {\em Convergence theory of nonlinear
  newton-krylov algorithms}, SIAM Journal on Optimization, 4 (1994),
  pp.~297--330.

\bibitem{elman1994inexact}
{\sc H.~C. Elman and G.~H. Golub}, {\em Inexact and preconditioned uzawa
  algorithms for saddle point problems}, SIAM Journal on Numerical Analysis, 31
  (1994), pp.~1645--1661.

\bibitem{fang2009two}
{\sc H.-r. Fang and Y.~Saad}, {\em Two classes of multisecant methods for
  nonlinear acceleration}, Numerical Linear Algebra with Applications, 16
  (2009), pp.~197--221.

\bibitem{feingold1962block}
{\sc D.~G. Feingold, R.~S. Varga, et~al.}, {\em Block diagonally dominant
  matrices and generalizations of the gerschgorin circle theorem}, Pacific J.
  Math, 12 (1962), pp.~1241--1250.

\bibitem{francca2015explicit}
{\sc G.~Fran{\c{c}}a and J.~Bento}, {\em An explicit rate bound for the
  over-relaxed admm}, arXiv preprint arXiv:1512.02063,  (2015).

\bibitem{ghadimi2015optimal}
{\sc E.~Ghadimi, A.~Teixeira, I.~Shames, and M.~Johansson}, {\em Optimal
  parameter selection for the alternating direction method of multipliers
  (admm): quadratic problems}, Automatic Control, IEEE Transactions on, 60
  (2015), pp.~644--658.

\bibitem{giselsson2014diagonal}
{\sc P.~Giselsson and S.~Boyd}, {\em Diagonal scaling in douglas-rachford
  splitting and admm}, in Decision and Control (CDC), 2014 IEEE 53rd Annual
  Conference on, IEEE, 2014, pp.~5033--5039.

\bibitem{greenbaum1997iterative}
{\sc A.~Greenbaum}, {\em Iterative methods for solving linear systems},
  vol.~17, Siam, 1997.

\bibitem{he2000alternating}
{\sc B.~He, H.~Yang, and S.~Wang}, {\em Alternating direction method with
  self-adaptive penalty parameters for monotone variational inequalities},
  Journal of Optimization Theory and applications, 106 (2000), pp.~337--356.

\bibitem{nesterov2004introductory}
{\sc Y.~Nesterov}, {\em Introductory lectures on convex optimization}, vol.~87,
  Springer Science \& Business Media, 2004.

\bibitem{nishihara2015general}
{\sc R.~Nishihara, L.~Lessard, B.~Recht, A.~Packard, and M.~I. Jordan}, {\em A
  general analysis of the convergence of admm}, arXiv preprint
  arXiv:1502.02009,  (2015).

\bibitem{oliveira2005new}
{\sc A.~R. Oliveira and D.~C. Sorensen}, {\em A new class of preconditioners
  for large-scale linear systems from interior point methods for linear
  programming}, Linear Algebra and its applications, 394 (2005), pp.~1--24.

\bibitem{rivlin1974chebyshev}
{\sc T.~J. Rivlin}, {\em The Chebyshev Polynomials: From Approximation Theory
  to Algebra and Number Theory}, John Wiley \& Sons, 1974.

\bibitem{saad1993flexible}
{\sc Y.~Saad}, {\em A flexible inner-outer preconditioned gmres algorithm},
  SIAM Journal on Scientific Computing, 14 (1993), pp.~461--469.

\bibitem{saad2003iterative}
{\sc Y.~Saad}, {\em Iterative methods for sparse linear systems}, Siam, 2003.

\bibitem{saad1986gmres}
{\sc Y.~Saad and M.~H. Schultz}, {\em Gmres: A generalized minimal residual
  algorithm for solving nonsymmetric linear systems}, SIAM Journal on
  scientific and statistical computing, 7 (1986), pp.~856--869.

\bibitem{spielman2014nearly}
{\sc D.~A. Spielman and S.-H. Teng}, {\em Nearly linear time algorithms for
  preconditioning and solving symmetric, diagonally dominant linear systems},
  SIAM Journal on Matrix Analysis and Applications, 35 (2014), pp.~835--885.

\bibitem{stuben2001review}
{\sc K.~St{\"u}ben}, {\em A review of algebraic multigrid}, Journal of
  Computational and Applied Mathematics, 128 (2001), pp.~281--309.

\bibitem{sturm1999using}
{\sc J.~F. Sturm}, {\em Using sedumi 1.02, a matlab toolbox for optimization
  over symmetric cones}, Optimization methods and software, 11 (1999),
  pp.~625--653.

\bibitem{toh2004solving}
{\sc K.-C. Toh}, {\em Solving large scale semidefinite programs via an
  iterative solver on the augmented systems}, SIAM Journal on Optimization, 14
  (2004), pp.~670--698.

\bibitem{trottenberg2000multigrid}
{\sc U.~Trottenberg, C.~W. Oosterlee, and A.~Schuller}, {\em Multigrid},
  Academic press, 2000.

\bibitem{vanderbei1995symmetric}
{\sc R.~J. Vanderbei}, {\em Symmetric quasidefinite matrices}, SIAM Journal on
  Optimization, 5 (1995), pp.~100--113.

\bibitem{vanderbei1993symmetric}
{\sc R.~J. Vanderbei and T.~J. Carpenter}, {\em Symmetric indefinite systems
  for interior point methods}, Mathematical Programming, 58 (1993), pp.~1--32.

\bibitem{vishnoi2012laplacian}
{\sc N.~K. Vishnoi}, {\em Laplacian solvers and their algorithmic
  applications}, Theoretical Computer Science, 8 (2012), pp.~1--141.

\bibitem{wang2001decomposition}
{\sc S.~Wang and L.~Liao}, {\em Decomposition method with a variable parameter
  for a class of monotone variational inequality problems}, Journal of
  optimization theory and applications, 109 (2001), pp.~415--429.

\bibitem{zhang2016on}
{\sc R.~Y. Zhang and J.~K. White}, {\em On the convergence of
  {GMRES}-accelerated {ADMM} in {$O(\kappa^{1/4}\log\epsilon^{-1})$} iterations
  for quadratic objectives}, arXiv preprint arXiv:1601.06200v3,  (2016).

\end{thebibliography}

\end{document}